\documentclass{amsart}
\usepackage{amscd}
\usepackage{amsmath}
\usepackage{amssymb}
\usepackage{amsthm}
\usepackage{graphicx}
\newtheorem{theorem}{Theorem}[section]
\newtheorem{lemma}[theorem]{Lemma}
\newtheorem{proposition}[theorem]{Proposition}

\newtheorem{question}[theorem]{Question}
\newtheorem{problem}[theorem]{Problem}
\theoremstyle{definition}

\newtheorem{example}[theorem]{Example}

\newtheorem*{acknowledgement}{Acknowledgment}
\theoremstyle{remark}
\newtheorem{remark}[theorem]{Remark}
\numberwithin{equation}{section}
\DeclareMathOperator{\Tor}{Tor}

\DeclareMathOperator{\Ker}{Ker}
\DeclareMathOperator{\Image}{Im}

\DeclareMathOperator{\height}{height}

\DeclareMathOperator{\reg}{reg}

\DeclareMathOperator{\lcm}{lcm}

\DeclareMathOperator{\ind-match}{ind-match}
\DeclareMathOperator{\match}{match}

\DeclareMathOperator{\min-match}{min-match}

\begin{document}
%%%%%%%%%%%%%%%%%%%%%%%%%%%%%%%%%%%%%%%%
%% title
%%%%%%%%%%%%%%%%%%%%%%%%%%%%%%%%%%%%%%%%
\title[Dominating induced matchings of finite graphs]
{Dominating induced matchings of finite graphs and regularity of edge ideals}
%%%%%%%%%%%%%%%%%%%%%%%%%%%%%%%%%%%%%%%%
%%%%%%%%%%%%%%%%%%%%%%%%%%%%%%%%%%%%%%%%
%% Information for authors
%%%%%%%%%%%%%%%%%%%%%%%%%%%%%%%%%%%%%%%%
\author[T. Hibi]{Takayuki Hibi}
\address[Takayuki Hibi]{Department of Pure and Applied Mathematics,
Graduate School of Information Science and Technology,
Osaka University,
Toyonaka, Osaka 560-0043, Japan}
\email{hibi@math.sci.osaka-u.ac.jp}
\author[A. Higashitani]{Akihiro Higashitani}
\address[Akihiro Higashitani]{
Department of Mathematics, Kyoto Sangyo University, 
Motoyama, Kamigamo, Kita-ku, Kyoto 603-8555, Japan}
\email{ahigashi@cc.kyoto-su.ac.jp}
\author[K. Kimura]{Kyouko Kimura}
\address[Kyouko Kimura]{Department of Mathematics, Faculty of Science, 
Shizuoka University,
836 Ohya, Suruga-ku, Shizuoka 422-8529, Japan}
\email{kimura.kyoko.a@shizuoka.ac.jp}
\author[A. Tsuchiya]{Akiyoshi Tsuchiya}
\address[Akiyoshi Tsuchiya]{Department of Pure and Applied Mathematics,
Graduate School of Information Science and Technology,
Osaka University,
Toyonaka, Osaka 560-0043, Japan}
\email{a-tsuchiya@cr.math.sci.osaka-u.ac.jp}
%\thanks{
%{\bf 2010 Mathematics Subject Classification:}
%Primary 05E40; Secondly 05C69, 05C70. \\
%%
%\hspace{0.43cm}
%{\bf Keywords:}
%edge ideal, regularity, dominating induced matching, unmixed graph, vertex decomposable graph}
%% \, \, \, \\
%% The fourth author had summer support provided by the JSPS Research
%% Fellowships for Young Scientists and the NSF East Asia and Pacific Institutes Fellowship. 
%}

%%%%%%%%%%%%%%%%%%%%%%%%%%%%%%%%%%%%%%%%
%% General info
%%%%%%%%%%%%%%%%%%%%%%%%%%%%%%%%%%%%%%%%
%\date{\today}
%\dedicatory{}

%%%%%%%%%%%%%%%%%%%%%%%%%%%%%%%%%%%%%%%%
%%%%%%%%%%%%%%%%%%%%%%%%%%%%%%%%%%%%%%%%
\begin{abstract}
The regularity of an edge ideal of a finite simple graph $G$
is at least the induced matching number of $G$ and
is at most the minimum matching number of $G$.
If $G$ possesses a dominating induced matching, i.e., 
an induced matching which forms a maximal matching,
then the induced matching number of $G$ is equal to
the minimum matching number of $G$.
In the present paper, from viewpoints of both combinatorics
and commutative algebra, finite simple graphs with
dominating induced matchings will be mainly studied.
\end{abstract}
\maketitle
%\thispagestyle{empty}
%%%%%%%%%%%%%%%%%%%%%%%%%%%%%%%%%%%%%%%%

%%%%%%%%%%%%%%%%%%%%%%%%%%%%%%%%%%%%%%%%%%%%%%%%
\section*{Introduction}
The regularity of an edge ideal of a finite simple graph has been studied
by many articles including \cite{BC1205}, \cite{BC1302}, \cite{DHS}, 
%\cite{FHVT}, 
%\cite{HaTuyl07}, 
\cite{HaVThypergraph}, \cite{Katzman}, \cite{KAM}, 
\cite{Kummini}, \cite{MMCRTY}, 
%\cite{MYPZN}, 
%\cite{MV}, 
\cite{Nevo}, 
\cite{VanTuyl}, \cite{Woodroofe-regularity} and  \cite{Zheng}. 
Recall that a finite graph is {\em simple} if 
it possesses no loop and no multiple edge.

Let $G$ be a finite simple graph on the vertex set
$[n] = \{ 1, \ldots, n \}$ with the edge set $E(G)$
and $S = K[x_{1}, \ldots, x_{n}]$ 
the polynomial ring in $n$ variables over a field $K$ with standard grading. 
The {\em edge ideal} of $G$ is the ideal $I(G) \subset S$ which is generated
by those squarefree quadratic monomials $x_{i}x_{j}$ with $\{ i, j \} \in E(G)$.
Following the previous paper \cite{HHKO}, we continue our research 
on the relation between the regularity $\reg(S/I(G))$ of
the quotient ring $S/I(G)$ and the matching number, the minimum matching number 
together with the induced matching number of $G$.

A {\em matching} of $G$ is a subset $\mathcal{M} \subset E(G)$ such that, 
for $e$ and $e'$ belonging to $\mathcal{M}$ with $e \neq e'$, one has $e \cap e'
= \emptyset$.  A maximal matching of $G$ is a matching $\mathcal{M}$ of $G$
for which $\mathcal{M} \cup \{ e \}$ cannot be a matching of $G$
for all $e \in E(G) \setminus \mathcal{M}$.  An {\em induced matching} is
a matching $\mathcal{M}$ of $G$ such that, for $e$ and $e'$ belonging to 
$\mathcal{M}$ with $e \neq e'$, there is no edge $f \in E(G)$ with
$e \cap f \neq \emptyset$ and $e' \cap f \neq \emptyset$.
The {\em matching number} of $G$, denoted by $\match(G)$, is 
the maximum cardinality of the matchings of $G$ and 
the {\em minimum matching number} of $G$, denoted by $\min-match(G)$,
is the minimum cardinality of the maximal matchings of $G$.
Furthermore, the {\em induced matching number} of $G$,
denoted by $\ind-match(G)$, is the maximum cardinality of
the induced matching of $G$.

The basic inequalities, due to \cite{Katzman} and 
\cite{Woodroofe-regularity}, among the above three invariants
% $\match(G), \min-match(G), \ind-match(G)$ 
together with $\reg(S/I(G))$ 
are 
\[
  \ind-match(G) \leq \reg S/I(G) \leq \min-match(G) \leq \match(G). 
\]
In addition, one can easily prove the inequality 
\[
\match(G) \leq 2 \min-match(G),
\] 
see Proposition \ref{claim:m<2m}. 
Naturally, one question arises: Given integers $p, c, q, r$ satisfying
\[
0 < p \leq c \leq q \leq r \leq 2q,
\]  
we can ask if there exists a finite simple graph $G$ for which
\[
\ind-match(G) = p, \, \,  
\reg S/I(G) = c, \, \,   
\min-match(G) = q, \, \,  
\match(G) = r.
\]
In Section $1$, this question and its related
problems will be studied.

Cameron and Walker \cite{CW} succeeded in characterizing a finite simple
graph $G$ with $\ind-match(G) = \match(G)$.  
For example, if $G$ is a star or a star triangle, then  
one has $\ind-match(G) = \match(G)$.
We say that a finite connected simple graph $G$ is a {\em Cameron--Walker graph}
if $\ind-match(G) = \match(G)$ and if $G$ is neither a star nor a star triangle.
Thus in particular for a Cameron--Walker graph $G$, one has
\[
  \ind-match(G) = \reg S/I(G) = \min-match(G) = \match(G). 
\] 
From a viewpoint of commutative algebra, the study on Cameron--Walker 
graphs is done in \cite{HHKO}.
In Section $2$, we treat some classes of finite simple graphs which contain 
Cameron--Walker graphs as a subclass and investigate these combinatorial 
properties. 

A {\em dominating induced matching} of $G$ is an induced matching
which also forms a maximal matching of $G$.
Every Cameron--Walker graph possesses a dominating induced matching.
Clearly a finite simple graph $G$ with
a dominating induced matching satisfies the equalities
\[
\ind-match(G) = \reg S/I(G) = \min-match(G).
\]
However, there is a finite simple graph $G$ which
possesses no dominating induced matching, but 
satisfies the equality $\ind-match(G) = \min-match(G)$. 
A characterization of finite simple graphs 
possessing dominating induced matchings
%is obtained by Lin, Mizrahi and Szwarcfiter \cite{LMS}.
is easy, see Proposition \ref{ind-match=max-match}. 

% In Section $2$, various questions on induced matchings will be discussed.
Our first work is to find a characterization of finite simple graphs $G$ 
satisfying $\ind-match(G) = \min-match(G)$ 
(Theorem \ref{ind-match=min-match}). 

Recall that a {\em vertex cover} of a finite simple graph $G$ on $[n]$ 
is a subset $C \subset [n]$ for which $C \cap e \neq \emptyset$
for all $e \in E(G)$.  A {\em minimal vertex cover} of $G$ is a vertex
cover $C$ of $G$ for which no proper subset of $C$ 
can be a vertex cover of $G$. 
A finite simple graph $G$ is called {\em unmixed} if 
all minimal vertex covers have the same cardinality.
Our second work is to characterize unmixed graphs 
with dominating induced matchings (Theorem \ref{claim:unmixedDIM}).

Finally, in Section $3$, the algebraic study of finite simple graphs
with dominating induced matchings will be discussed.
In \cite{HHKO} it is shown that every Cameron--Walker graph is
vertex decomposable, hence sequentially Cohen--Macaulay.  However,
there is a finite simple graph $G$ with a dominating induced matching
such that $G$ is not sequentially Cohen--Macaulay.
We cannot escape from the temptation to find a characterization
of vertex decomposable graphs with dominating induced matchings.
However, to find a complete characterization seems to be rather difficult.
We try to find a class ${\mathcal A}$ of vertex decomposable graphs 
with dominating induced matchings such that ${\mathcal A}$ contains 
all Cameron--Walker graphs.
In addition, various examples will be supplied.

\section{Matching number, induced matching number, and regularity}
Let $G$ be a finite simple graph. 
A {\em matching} of $G$ is a subset $\mathcal{M} \subset E(G)$ such that 
$e \cap e' = \emptyset$ for all $e, e' \in \mathcal{M}$ with $e \neq e'$. 
We denote by $\match (G)$ (resp.\  $\min-match (G)$), 
the maximum (resp.\  minimum) cardinality among maximal matchings of $G$. 
Two edges $e, e' \in E(G)$ are said to be {\em $3$-disjoint} if 
$e \cap e' = \emptyset$ and there is no edge $f \in E(G)$ with 
$e \cap f \neq \emptyset$ and $e' \cap f \neq \emptyset$. 
An {\em induced matching} is a set of edges which are pairwise $3$-disjoint. 
We denote by $\ind-match (G)$, the maximum cardinality among 
induced matchings of $G$. 
By Katzman \cite{Katzman} and Woodroofe \cite{Woodroofe-regularity}, 
we have 
\begin{equation}
  \label{eq:all-ineq}
  \ind-match (G) \leq \reg S/I(G) \leq \min-match (G) \leq \match (G). 
\end{equation}
In this section, we investigate the problem to construct a finite simple 
connected graph with given values of these $4$ invariants. 

\par
We first note the relation between $\match (G)$ and $\min-match (G)$. 
\begin{proposition}
  \label{claim:m<2m}
  Let $G$ be a finite simple graph. 
  Then $\match (G) \leq 2 \min-match (G)$. 
\end{proposition}
\begin{proof}
  Let $\{u_i, v_i \}$, $i = 1, 2, \ldots, q$ be edges of $G$ which 
  form a maximal matching with $q = \min-match (G)$. 
%  Also let $w_1, \ldots, w_s$ be the rest vertices. 
  Let $e$ be an edge in $G$. Then $e$ contains at least one vertex of 
  $2q$ vertices $u_1, v_1, \ldots, u_q, v_q$. 
  Therefore there is no matching which consists of $2q+1$ edges. 
\end{proof}

Then the following problem naturally occurs: 
\begin{problem}
  Let $p, c, q, r$ be integers satisfying 
  \begin{displaymath}
    0 < p \leq c \leq q \leq r \leq 2q. 
  \end{displaymath}
  Construct a finite simple connected graph $G$ satisfying 
  \begin{displaymath}
      \ind-match (G) = p, \quad 
      \reg S/I(G) = c, \quad
      \min-match (G) = q, \quad
      \match (G) = r. 
%    \begin{aligned}
%      \ind-match (G) &= p, \\
%      \reg S/I(G) &= c, \\
%      \min-match (G) &= q, \\
%      \match (G) &= r. 
%    \end{aligned}
  \end{displaymath}
\end{problem}

When we ignore the condition for the regularity, we can do. 
The following result might be known, however we give a proof of it 
for the sake of completeness. 
\begin{theorem}
  \label{ind-min-match}
  For arbitrary integers $p, q, r$ with $0 < p \leq q \leq r \leq 2q$, 
  there exists a finite simple connected graph $G$ which satisfies 
  \begin{displaymath}
    \ind-match (G) = p, \qquad \min-match (G) = q, \qquad \match (G) = r. 
  \end{displaymath}
\end{theorem}
%\begin{remark}
%  For a graph $G$, the inequality $\match (G) \leq 2 \min-match (G)$ holds. 
%
%  \par
%  Actually, let $\{u_i, v_i \}$, $i = 1, 2, \ldots, r$ be edges of $G$ which 
%  form maximal matching with $r = \min-match (G)$. 
%%  Also let $w_1, \ldots, w_s$ be the rest vertices. 
%  Let $e$ be an edge in $G$. Then $e$ contains at least one vertex of 
%  $2r$ vertices $u_1, v_1, \ldots, u_r, v_r$. 
%  Therefore there is no matching which consists of $2r+1$ edges. 
%\end{remark}
%\begin{proof}[Proof of Theorem \ref{ind-min-match}]
\begin{proof}
  Let $a, b, m, n$ be non-negative integers with $m \leq n$ and $1 \leq n$. 
  Let us consider the following graph $G_{a,b,m,n}$: 
% which consists of $K_{2n}$ attaching to 
  \newline
  \begin{picture}(400,270)(10,-70)
    \put(50,170){$G_{a,b,m,n}$:}
    \put(200,100){\oval(100,50)}
    \put(170,125){\circle*{5}}
    \put(190,125){\circle*{5}}
    \put(205,125){\circle*{2}}
    \put(210,125){\circle*{2}}
    \put(215,125){\circle*{2}}
    \put(205,150){\circle*{2}}
    \put(210,150){\circle*{2}}
    \put(215,150){\circle*{2}}
    \put(230,125){\circle*{5}}
    \put(200,75){\circle*{5}}
    \put(170,175){\circle*{5}}
    \put(190,175){\circle*{5}}
    \put(230,175){\circle*{5}}
    \put(150,110){\circle*{5}}
    \put(150,95){\circle*{2}}
    \put(150,90){\circle*{2}}
    \put(152,85){\circle*{2}}
    \put(170,75){\circle*{5}}
    \put(250,110){\circle*{5}}
    \put(250,95){\circle*{2}}
    \put(250,90){\circle*{2}}
    \put(248,85){\circle*{2}}
    \put(230,75){\circle*{5}}
    %%%%%%%%%%%%%%%%%%%%%%%%%%%
    \put(50,25){\circle*{5}}
    \put(100,25){\circle*{5}}
    \put(175,25){\circle*{5}}
    \put(225,25){\circle*{5}}
    \put(300,25){\circle*{5}}
    \put(350,25){\circle*{5}}
    %%%%%%%%%%%%%%%%%%%%%%%%%%%
    \put(50,-5){\circle*{5}}
    \put(50,-35){\circle*{5}}
    \put(30,5){\circle*{5}}
    \put(30,-25){\circle*{5}}
    \put(100,-5){\circle*{5}}
    \put(100,-35){\circle*{5}}
    \put(80,5){\circle*{5}}
    \put(80,-25){\circle*{5}}
    \put(175,-5){\circle*{5}}
    \put(175,-35){\circle*{5}}
    \put(155,5){\circle*{5}}
    \put(155,-25){\circle*{5}}
    \put(225,-5){\circle*{5}}
    \put(225,-35){\circle*{5}}
    \put(300,-5){\circle*{5}}
    \put(300,-35){\circle*{5}}
    \put(350,-5){\circle*{5}}
    \put(350,-35){\circle*{5}}
    \put(120,5){\circle*{2}}
    \put(125,5){\circle*{2}}
    \put(130,5){\circle*{2}}
    \put(255,5){\circle*{2}}
    \put(260,5){\circle*{2}}
    \put(265,5){\circle*{2}}
    %%%%%%%%%%%%%%%%%%%%%%%%%%%
    \put(170,125){\line(0,1){50}}
    \put(190,125){\line(0,1){50}}
    \put(230,125){\line(0,1){50}}
    %%%%%%%%%%%%%%%%%%%%%%%%%%%
    \put(200,75){\line(-3,-1){150}}
    \put(200,75){\line(-2,-1){100}}
    \put(200,75){\line(-1,-2){25}}
    \put(200,75){\line(1,-2){25}}
    \put(200,75){\line(2,-1){100}}
    \put(200,75){\line(3,-1){150}}
    %%%%%%%%%%%%%%%%%%%%%%%%%%%
    \put(50,25){\line(0,-1){60}}
    \put(50,25){\line(-1,-1){20}}
    \put(50,-5){\line(-1,-1){20}}
    \put(100,25){\line(0,-1){60}}
    \put(100,25){\line(-1,-1){20}}
    \put(100,-5){\line(-1,-1){20}}
    \put(175,25){\line(0,-1){60}}
    \put(175,25){\line(-1,-1){20}}
    \put(175,-5){\line(-1,-1){20}}
    \put(225,25){\line(0,-1){60}}
    \put(300,25){\line(0,-1){60}}
    \put(350,25){\line(0,-1){60}}
    %%%%%%%%%%%%%%%%%%%%%%%%%%%
    \put(25,-40){$\underbrace{\phantom{aaaaaaaaaaaaaaaaaaaaaaaaaaaaaa}}$}
    \put(100,-55){$a$}
    \put(215,-40){$\underbrace{\phantom{aaaaaaaaaaaaaaaaaaaaaaaaaaa}}$}
    \put(280,-55){$b$}
    \put(165,177){$\overbrace{\phantom{aaaaaaaaaaaaa}}$}
    \put(195,190){$2m$}
    \put(190,100){$K_{2n}$}
  \end{picture}
  \newline
  Then $G_{a,b,m,n}$ satisfies 
  \begin{displaymath}
    \begin{aligned}
      \ind-match (G_{a,b,m,n}) &= a+b+1, \\
      \min-match (G_{a,b,m,n}) &= a+b+n, \\
      \match (G_{a,b,m,n}) &= 2a+b+(n-m) + 2m = 2a + b + n + m. 
    \end{aligned}
  \end{displaymath}
  Therefore we obtain a desired graph $G$ if we can choose $a, b, n, m$ 
  satisfying 
  \begin{displaymath}
    \begin{aligned}
      a+b+1 &= p, \\
      a+b+n &= q, \\
      2a+b+n+m &=r. 
    \end{aligned}
  \end{displaymath}

  \par
  Indeed, we can choose such $a,b,m,n$. 
  First, by $a+b=p-1$, we take $n = q-(a+b) = q-p+1 > 0$. 
  Then 
  \begin{displaymath}
      r = 2a+b+n+m = a+ (p-1) + (q-p+1) + m, 
  \end{displaymath}
  and we have $a+m = r-q$. 
  Note that $a+b=p-1$ and $a+m=r-q$. 

  \par
  \textit{Case 1}: $r-q \leq q-p+1$. 
  We can take $a=0$, $b=p-1$, and $m=r-q$. (Then $m \leq n$.) 

  \par
  \textit{Case 2}: $r-q > q-p+1$. 
  We set $m=q-p+1$ and $a=r-2q+p-1 (>0)$. 
  Then $m=n$ and $b = (p-1) - (r-2q+p-1) = 2q-r \geq 0$. 
  The last inequality follows from the condition $r \leq 2q$. 
\end{proof}

%\begin{problem}
%  Let $p \leq r \leq q$ be positive integers. 
%  Then is there any finite simple connected graph $G$ which satisfies  
%  \begin{displaymath}
%    \min-match (G) = q, \qquad \reg (S/I(G)) = r, \qquad \ind-match (G) = p? 
%  \end{displaymath}
%\end{problem}
%(On the above problem, we allow that the number of vertices of $G$ 
% is no restriction)

%\begin{proposition}[{cf.\  \cite[p.4, Figure 1]{BC1302}}]
%  For an arbitrary positive integer $p'$, there exists a connected graph 
%  $G$ which satisfies $\ind-match (G) = \reg S/I(G) - p'$. 
%  Actually, the graph which consists of a complete bipartite graph 
%  $K_{1, p'+1}$ and $5$-cycles attaching each $p' + 1$ vertices 
%  is an example for such a graph $G$. 
%\end{proposition}

Also, the difference between the regularity and the induced matching number 
as well as the difference between the minimum matching number 
and the regularity can be arbitrary large. 

\par
Recall that the regularity of a (standard graded) $S$-module $M$ is defined by 
\begin{displaymath}
  \reg (M) = \max \{ j-i \; : \; \beta_{ij} (M) \neq 0 \}, 
\end{displaymath}
where $\beta_{ij} (M) := \dim_K [\Tor_i^S (K, M)]_j$, 
%$= \dim_K [\Ker (1 \otimes d_i)/\Image (1 \otimes d_{i+1})]_j$, 
the $ij$th Betti number of $M$. 

\begin{theorem}
  \label{CompleteBipartite5cycle4cycle}
  For arbitrary non-negative integers $a, b$, there exists a finite 
  simple connected graph $G$ satisfying 
  \begin{equation}
    \label{eq:distance-reg}
    \begin{aligned}
      \ind-match (G) &= \reg (S/I(G)) - a, \\
      \min-match (G) &= \reg (S/I(G)) + b. 
%      \match (G) &= \reg (S/I(G)) + b + 1. 
    \end{aligned}
  \end{equation}
\end{theorem}
\begin{remark}
  Although for a given integer $c$, there exists a simple connected graph 
  $G$ with $\reg S/I(G) = c$ (for example, the cycle of length $3c-1$ 
  is such a graph), 
  we do not know whether there exists a finite simple connected graph 
  $G$ satisfying $\reg S/I(G) = c$ together with (\ref{eq:distance-reg}) 
  for given integers $a,b,c$. 
\end{remark}

Let $G$ be a finite simple graph on $V$. 
When we identify the vertices of $G$ with the variables of the underlying 
polynomial ring $S$ of the edge ideal $I(G)$, we denote $S=K[V]$. 
%We denote by $K[V]$ the polynomial ring over $K$ whose variables are 
%$v \in V$. 
Theorem \ref{CompleteBipartite5cycle4cycle} immediately follows 
by the following lemma. 
%For the above graph, $\min-match (G) = \match (G) - 1$ holds. 
%Therefore $\min-match (G) = \reg S/I(G) + q'$. 
%Also this graph $G$ satisfies 
\begin{lemma}
  \label{CompleteBipartite5cycle4cycle-invariants}
  Let $a, b$ be non-negative integers. 
  Let $G_{a,b}$ be the graph consisting of a complete bipartite graph 
  $K_{1, a+b+1}$ with the bipartition 
  $\{ x \} \sqcup \{ y_1, \ldots, y_{a+b+1} \}$
  and $5$-cycles attaching to each $y_1, \ldots, y_{a+1}$ 
  and $4$-cycles attaching to each $y_{a+2}, \ldots, y_{a+b+1}$. 
  We denote by $V_{a,b}$, the vertex set of $G_{a,b}$. 
  Then 
%  \begin{displaymath}
%    \begin{aligned}
%      \ind-match (G_{a,b}) &= \reg (S/I(G)) - p', \\
%      \match (G_{a,b}) &= \reg (S/I(G)) + q' + 1. 
%    \end{aligned}
%  \end{displaymath}
  \begin{displaymath} 
    \begin{aligned}
      \ind-match (G_{a,b}) &= a + b +2, \\
      \min-match (G_{a,b}) &= 2a + 2b + 2, \\
      \match (G_{a,b}) &= 2a + 2b + 3, \\
      \reg (K[V_{a,b}]/I(G_{a,b})) &= 2a + b + 2. 
    \end{aligned}
  \end{displaymath}
\end{lemma}

In order to prove Lemma \ref{CompleteBipartite5cycle4cycle-invariants}, 
we use the following two results. 
\begin{lemma}[{Woodroofe \cite[Corollary 10]{Woodroofe-regularity}}]
  \label{Woodroofe:regularity}
  If a graph $G$ has an induced subgraph $H$ which consists of 
  disjoint union of $m$ edges and cycles 
  $C_{3 i_1 + 2}, \ldots, C_{3 i_n + 2}$, 
  then $\reg S/I(G) \geq m + n + \sum_{j=1}^n i_j$. 
\end{lemma}
\begin{lemma}[{Kalai and Meshulam \cite[Theorem 1.2]{KM}}]
  \label{KM:regularity}
  Let $I_1, \ldots, I_s$ be squarefree monomial ideals of $S$. 
  Then 
  \begin{displaymath}
    \reg S/(I_1 + \cdots + I_s) \leq \sum_{j=1}^s \reg S/I_j. 
  \end{displaymath}
\end{lemma}

\begin{proof}[Proof of Lemma \ref{CompleteBipartite5cycle4cycle-invariants}]
%  We set the vertex set of the subgraph $K_{1, a+b+1}$ of $G$ as 
%  $\{ x, y_1, \ldots, y_{a+b+1} \}$ so that 
%  $\{ x, y_i \} \in E(K_{1, a+b+1})$. 
%  Also assume that a $5$-cycle is attached to $y_i$ for $i=1, \ldots, a+1$, 
%  and $4$-cycle is attached to $y_i$ for $i = a+2, \ldots, a+b+1$. 
%
%  \par
  We first compute $\ind-match (G_{a,b})$. 
  Note that we cannot choose $2$ edges which are $3$-disjoint in $G_{a,b}$ 
  from each $4$-cycle or each $5$-cycle. 
  The same is true for $K_{1, a+b+1}$. 
  Therefore $\ind-match (G_{a,b}) \leq a+b+2$. 
  Indeed, there exist $a+b+2$ edges of $G_{a,b}$ which
  form an induced matching of $G_{a,b}$: 
  we choose an edge which does not contain $y_i$ from each $4$-cycle, 
  the edge which is $3$-disjoint with $\{ x, y_i \}$ from each 
  $5$-cycle, and $\{ x, y_1 \}$. 

  \par
  We next compute $\reg K[V_{a,b}]/I(G_{a,b})$. 
  Take an edge from each $4$-cycle. 
  Then the graph which consists of 
  these $b$ edges and $a + 1$ copies of $5$-cycles is an induced subgraph of 
  $G_{a,b}$. 
%  The set of these edges form an induced subgraph of $G_{a,b}$. 
  Therefore by Lemma \ref{Woodroofe:regularity}, we have
  \begin{displaymath}
    \reg (K[V_{a,b}]/I(G_{a,b})) \geq b + 2(a+1). 
  \end{displaymath}

  \par
  In order to prove the opposite inequality, 
%  On the other hand, we prove $\reg S/I(G_{a,b}) \leq 2(a+1) + b$. 
  we define subgraphs $G_1, \ldots, G_{a+b+1}$ of $G_{a,b}$. 
  For $1 \leq i \leq a+1$, let $G_i$ be the subgraph of $G_{a,b}$ 
  consisting of the $5$-cycle containing $y_i$ and $\{ x, y_i \}$. 
  Also for $a+2 \leq i \leq a+b+1$, 
  let $G_i$ be the subgraph of $G_{a,b}$ consisting of the $4$-cycle 
  containing $y_i$ and $\{ x, y_i \}$. 
  Then $E(G_{a,b}) = \bigcup_{i=1}^{a+b+1} E(G_i)$. 
  Since 
  \begin{displaymath}
    \reg (K[V_{a,b}]/I(G_i)K[V_{a,b}]) = \left\{ 
    \begin{alignedat}{3}
      &2, &\quad &1 \leq i \leq a+1, \\
      &1, &\quad &a+2 \leq i \leq a+b+1, 
    \end{alignedat}
    \right. 
  \end{displaymath}
  we have $\reg (K[V_{a,b}]/I(G_{a,b})) \leq 2(a+1)+b$ 
  by Lemma \ref{KM:regularity}. 

  \par
  Therefore $\reg (K[V_{a,b}]/I(G_{a,b})) = b + 2(a+1)$ holds. 

  \par
  Finally we compute $\min-match (G_{a,b})$ and $\match (G_{a,b})$. 
  Note that 
  \begin{displaymath}
    \begin{aligned}
      &\min-match (C_4) = \match (C_4) = 2, \\
      &\min-match (C_5) = \match (C_5) = 2, \\
      &\min-match (K_{1, a+b+1}) = \match (K_{1, a+b+1}) = 1. 
    \end{aligned}
  \end{displaymath}
  Let $\mathcal{M}$ be a maximal matching of $G$. 
  If $\{ x, y_i \} \notin \mathcal{M}$ for $i = 1, \ldots, a+b+1$, 
  then we have $\# \mathcal{M}  = 2(a+1) + 2b$. 
  If $\{ x, y_i \} \in \mathcal{M}$ for some $1 \leq i \leq a+1$, 
  then the cardinality of $\mathcal{M}$ is either $1 + 2(a+1) + 2b$ or 
  $1 + 1 + 2a + 2b = 2(a+1) + 2b$. 
  If $\{ x, y_i \} \in \mathcal{M}$ for some $a+2 \leq i \leq a+b+1$, 
  then we have $\# \mathcal{M} = 2(a+1) + 2b$. 
  Therefore we have the desired assertions. 
%  If we choose maximal matching of each $5$-cycle which contains 
%  $y_i$, then we obtain maximal matching of $G_{a,b}$ with cardinality 
%  $2(a+1)+2b$ which gives the minimum. 
%  Otherwise, we obtain a maximal matching of $G_{a,b}$ 
%  with cardinality $2(a+1) + 2b + 1$. 
\end{proof}

%Note that inequalities $\reg (S/I(G)) \leq \co-chord (G) \leq m(G)$ , 
%where $\co-chord (G)$ denotes the co-chordal cover number of $G$ 
%(Woodroofe \cite{Woodroofe-regularity}). 
%\begin{proposition}[{\cite[p.\  17, Proposition 3.12 (Figure 5)]{BC1205}}]
%  For an arbitrary positive integer $q' \geq 1$, 
%  there exists a connected graph $G$ with 
%  $\co-chord (G) = \reg S/I(G) + q'$. 
%\end{proposition}
%
%\begin{proposition}[{\cite[Proposition 8]{Woodroofe-regularity}}]
%  For arbitrary non-negative integers $p', q'$, 
%  there exists a graph $G$ with 
%  \begin{displaymath}
%    \begin{aligned}
%      \ind-match (G) &= \reg (S/I(G)) - p', \\
%      \co-chord (G) &= \reg (S/I(G)) + q'. 
%    \end{aligned}
%  \end{displaymath}
%  Actually, the graph $G$ consists of $p'$ copies of $5$-cycles 
%  and $q'$ copies of $7$-cycles 
%  is such an graph. In this case, $\co-chord (G) = \match (G)$ holds. 
%\end{proposition}
%
%
%\begin{problem}
%  For arbitrary non-negative integers $p,q$, 
%  are there exists a finite simple connected graph $G$ which satisfies 
%  \begin{displaymath}
%    \ind-match (G) = \reg (S/I(G)) - p,
%   \qquad \match (G) = \reg (S/I(G)) + q? 
%  \end{displaymath}
%\end{problem}

\par
Now we return to the first inequalities (\ref{eq:all-ineq}). 
There are the following $8$ cases: 
\begin{enumerate}
\item[(i)] $\ind-match (G) = \reg S/I(G) = \min-match (G) = \match (G)$. 
\item[(ii)] $\ind-match (G) = \reg S/I(G) = \min-match (G) < \match (G)$. 
\item[(iii)] $\ind-match (G) = \reg S/I(G) < \min-match (G) = \match (G)$. 
\item[(iv)] $\ind-match (G) = \reg S/I(G) < \min-match (G) < \match (G)$. 
\item[(v)] $\ind-match (G) < \reg S/I(G) = \min-match (G) = \match (G)$. 
\item[(vi)] $\ind-match (G) < \reg S/I(G) = \min-match (G) < \match (G)$. 
\item[(vii)] $\ind-match (G) < \reg S/I(G) < \min-match (G) = \match (G)$. 
\item[(viii)] $\ind-match (G) < \reg S/I(G) < \min-match (G) < \match (G)$. 
\end{enumerate}

For each case, is there a finite simple connected graph $G$ 
satisfying the inequalities? 
The following theorem is an answer to the question. 
%We answer this question positively. 
\begin{theorem}
  \label{graph-ineq}
  There exists a finite simple connected graph $G$ satisfying the inequalities. 
  In particular, we can construct an infinite family of 
  finite simple connected graphs 
  satisfying each inequalities except for the case (v). 
%  For each cases (i), ..., (viii), 
%  there exists a graph $G$ satisfying the inequalities. 
%  In particular, we can construct an infinite sequence of the graph 
%  satisfying each cases except for the case (v). 
\end{theorem}
A graph $G$ is called chordal if any cycle in $G$ of length more than $3$ 
has a chord. H\`{a} and Van Tuyl \cite{HaVThypergraph} proved that 
$\reg S/I(G) = \ind-match (G)$ holds for a chordal graph $G$. 
\begin{proof}[Proof of Theorem \ref{graph-ineq}]
  \textit{Case (i):} 
  The Cameron--Walker graphs \cite{HHKO} are just such graphs $G$. 

  \par
  \textit{Case (ii):} 
  The path graph $P_{6n}$ with $6n$ vertices $(n \geq 1)$ 
  satisfies the inequalities. 
  Indeed, $\ind-match (P_{6n}) = \min-match (P_{6n}) =2n$ and 
  $\match (P_{6n}) = 3n$. 
  (Note that $P_{6n}$ has a dominating induced matching; 
  see Section \ref{sec:DIM}.) 

  \par
  \textit{Case (iii):} 
  The complete graph $K_n$ with $n$ vertices $(n \geq 4)$ 
  satisfies the inequalities. 
  Indeed, $\ind-match (K_n) = 1$ and 
  $\min-match (K_n) = \match (K_n) = \lfloor n/2 \rfloor \geq 2$. 
  Also, since $K_n$ is a chordal graph, it follows that 
  $\ind-match (K_n) = \reg (S/I(K_n))$ by \cite{HaVThypergraph}. 

  \par
  \textit{Case (iv):} 
  The fully whiskered graph $W(K_n)$ of the complete graph $K_n$ 
  $(n \geq 3)$ satisfies the inequalities. 
  Here $W(K_n)$ is defined as follows: 
  let $x_1, \ldots, x_n$ be vertices of $K_n$ and let 
  $y_1, \ldots, y_n$ be new vertices. 
  Then $W(K_n)$ is the graph on $\{ x_1, \ldots, x_n,\, y_1, \ldots, y_n \}$ 
  whose edge set is $E(K_n) \cup \{ \{ x_i, y_i \} : i=1, 2, \ldots, n \}$. 
  Note that $W(K_n)$ is also a chordal graph. 
  Thus $\ind-match (G) = \reg S/I(G)$ holds. 
%  Also, it is easy to see that 
  Since 
  $\ind-match (G) = 1$, 
  $\min-match (G) = \lceil n/2 \rceil$ and $\match (G) = n$, 
  $G$ satisfies the desired inequalities. 

  \par
  \textit{Case (v):} The $5$-cycle $C_5$ satisfies the desired inequalities. 
  Indeed, $\ind-match (C_5) = 1$ and 
  $\reg S/I(C_5) = \min-match (C_5) = \match (C_5) = 2$. 

  \par
  \textit{Cases (vi) and (viii)}: 
  Consider the graph $G_{a,b}$ 
  on Lemma \ref{CompleteBipartite5cycle4cycle-invariants}. 
  If $a \neq 0$ and $b=0$, then the graph satisfies the inequalities (vi). 
  If $a \neq 0$ and $b \neq 0$, then the graph satisfies the 
  inequalities (viii). 

  \par
  \textit{Cases (vii):} 
  For an integer $k \geq 2$, let $H_k$ be the graph on 
  \begin{displaymath}
    V_k := \{ u,v, \; x_1, x_2, \ldots, x_k, \; y_1, y_2, \ldots, y_k, \; 
       z_{11}, z_{12}, z_{21}, z_{22}, \ldots, z_{k1}, z_{k2} \} 
  \end{displaymath}
  with the following edges: 
  \begin{equation}
    \label{eq:label-edgeH_k}
    \begin{alignedat}{3}
      \mathfrak{e}_i &:= \{ u, x_i \}, &\  
      \mathfrak{f}_i &:= \{ v, y_i \}, \\
      \mathfrak{g}_{i1} &:= \{ x_i, z_{i1} \}, &\  
      \mathfrak{g}_{i2} &:= \{ x_i, z_{i2} \}, \\
      \mathfrak{h}_{i1} &:= \{ y_i, z_{i1} \}, &\  
      \mathfrak{h}_{i2} &:= \{ y_i, z_{i2} \}, 
    \end{alignedat}
    \quad i= 1, 2, \ldots, k; 
%;\{ v, y_2 \}, \ldots, \{ v, y_k \}, \\
%      &\{ x_1, z_{11} \}, \{ x_1, z_{12} \}, 
%       \{ x_2, z_{21} \}, \{ x_2, z_{22} \}, 
%         \ldots, \{ x_k, z_{k1} \}, \{ x_k, z_{k2} \}, \\
%      &\{ y_1, z_{11} \}, \{ y_1, z_{12} \}, 
%       \{ y_2, z_{21} \}, \{ y_2, z_{22} \}, 
%         \ldots, \{ y_k, z_{k1} \}, \{ y_k, z_{k2} \}; 
%    \end{alignedat}
%    \begin{aligned}
%      &\{ u, x_i \}, \{ u, x_2 \}, \ldots, \{ u, x_k \}, \ 
%       \{ v, y_1 \}, \{ v, y_2 \}, \ldots, \{ v, y_k \}, \\
%      &\{ x_1, z_{11} \}, \{ x_1, z_{12} \}, 
%       \{ x_2, z_{21} \}, \{ x_2, z_{22} \}, 
%         \ldots, \{ x_k, z_{k1} \}, \{ x_k, z_{k2} \}, \\
%      &\{ y_1, z_{11} \}, \{ y_1, z_{12} \}, 
%       \{ y_2, z_{21} \}, \{ y_2, z_{22} \}, 
%         \ldots, \{ y_k, z_{k1} \}, \{ y_k, z_{k2} \}; 
%    \end{aligned}
  \end{equation}
  see Figure \ref{fig:H_k}. 

  \par
  Then we see from Lemma \ref{H_k} below 
  that $H_k$ satisfies the inequalities (vii). 
\end{proof}

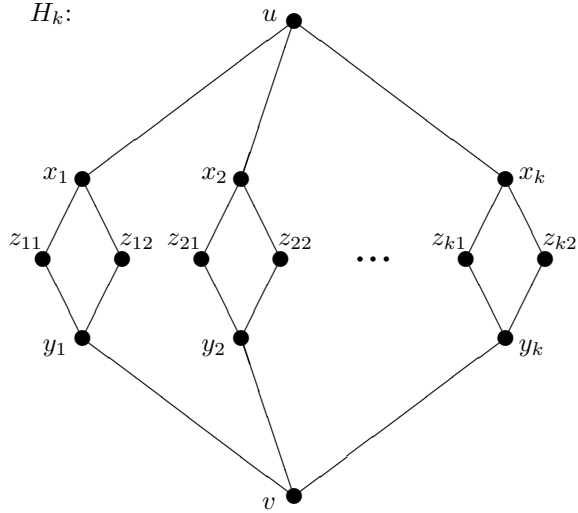
\begin{figure}[htbp]
  \begin{center}
    \begin{picture}(200,230)(-20,-105)
      \put(-20,90){$H_k$:}
%      \thicklines
      \put(-15,0){\circle*{6}}
      \put(15,0){\circle*{6}}
      \put(45,0){\circle*{6}}
      \put(75,0){\circle*{6}}
      \put(105,0){\circle*{2}}
      \put(110,0){\circle*{2}}
      \put(115,0){\circle*{2}}
      \put(145,0){\circle*{6}}
      \put(175,0){\circle*{6}}
      \put(0,30){\circle*{6}}
      \put(60,30){\circle*{6}}
      \put(160,30){\circle*{6}}
      \put(0,-30){\circle*{6}}
      \put(60,-30){\circle*{6}}
      \put(160,-30){\circle*{6}}
      \put(80,90){\circle*{6}}
      \put(80,-90){\circle*{6}}
      %%%%%%%%%
      \put(-28,5){$z_{11}$}
      \put(14,5){$z_{12}$}
      \put(32,5){$z_{21}$}
      \put(74,5){$z_{22}$}
      \put(132,5){$z_{k1}$}
      \put(174,5){$z_{k2}$}
      \put(-15,30){$x_{1}$}
      \put(45,30){$x_{2}$}
      \put(165,30){$x_{k}$}
      \put(-15,-35){$y_{1}$}
      \put(45,-35){$y_{2}$}
      \put(165,-35){$y_{k}$}
      \put(68,90){$u$}
      \put(68,-95){$v$}
      %%%%%%%%%
      \put(-15,0){\line(1,2){15}}
      \put(45,0){\line(1,2){15}}
      \put(145,0){\line(1,2){15}}
      \put(-15,0){\line(1,-2){15}}
      \put(45,0){\line(1,-2){15}}
      \put(145,0){\line(1,-2){15}}
      \put(15,0){\line(-1,2){15}}
      \put(75,0){\line(-1,2){15}}
      \put(175,0){\line(-1,2){15}}
      \put(15,0){\line(-1,-2){15}}
      \put(75,0){\line(-1,-2){15}}
      \put(175,0){\line(-1,-2){15}}
      \put(80,90){\line(-4,-3){80}}
      \put(80,90){\line(-1,-3){20}}
      \put(80,90){\line(4,-3){80}}
      \put(80,-90){\line(-4,3){80}}
      \put(80,-90){\line(-1,3){20}}
      \put(80,-90){\line(4,3){80}}
    \end{picture}
  \end{center}
  \caption{The graph $H_k$}
  \label{fig:H_k}
\end{figure}

\begin{lemma}
  \label{H_k}
  The graph $H_k$ ($k \geq 2$) in the proof of Theorem \ref{graph-ineq} 
  satisfies 
  \begin{displaymath}
    \begin{aligned}
      \ind-match (H_k) &= k, \\
      \min-match (H_k) &= \match (H_k) = 2k, \\
      \reg (K[V_k]/I(H_k)) &= k+1. 
    \end{aligned}
  \end{displaymath}
\end{lemma}

In order to prove Lemma \ref{H_k}, 
we use a Lyubeznik resolution (\cite{Lyubeznik}), 
which is a subcomplex of the Taylor resolution. 

\par
Let $I$ be a monomial ideal of $S$ 
and $m_1, \ldots, m_{\mu}$ the minimal monomial generators of $I$. 
The free basis $e_{i_1 \cdots i_s}$ of the Taylor resolution 
is said to be $L$-admissible if 
$\lcm (m_{i_t}, \ldots, m_{i_s})$ is not divisible by $m_q$ 
for all $1 \leq t < s$ and for all $q < i_t$. 
We will denote an $L$-admissible symbol $e_{i_1 \cdots i_s}$ 
by $[m_{i_1}, \ldots, m_{i_s}]$. 
The degree of an $L$-admissible symbol $[m_{i_1}, \ldots, m_{i_s}]$ 
is defined by the degree of $\lcm (m_{i_1}, \ldots, m_{i_s})$. 
(Recall that we consider the standard grading on the polynomial ring $S$.) 
An $L$-admissible symbol $[m_{i_1}, \ldots, m_{i_s}]$ is said to be maximal 
if there is no $L$-admissible symbol $[m_{j_1}, \ldots, m_{j_t}]$ 
such that $\{ i_1, \ldots, i_s \} \subsetneq \{ j_1, \ldots, j_t \}$. 
A Lyubeznik resolution $(\mathcal{L}_{\bullet}, d_{\bullet})$ of $I$ 
(with respect to the above order of the minimal monomial generators) 
is the subcomplex of the Taylor resolution 
generated by all $L$-admissible symbols, which is also a free resolution of 
$S/I$. 

%\par
%Recall that the regularity of $S/I$ is defined by 
%\begin{displaymath}
%  \reg (S/I) = \max \{ j-i \; : \; \beta_{ij} (S/I) \neq 0 \}, 
%\end{displaymath}
%where $\beta_{ij} (S/I) := \dim_K [\Tor_i^S (K, S/I)]_j$, 
%%$= \dim_K [\Ker (1 \otimes d_i)/\Image (1 \otimes d_{i+1})]_j$, 
%the $ij$th Betti number of $S/I$. 

\begin{proof}[Proof of Lemma \ref{H_k}]
  We first compute $\ind-match (H_k)$. 
  Let $\mathcal{M}$ be a maximal induced matching of $H_k$, i.e., 
  $\mathcal{M}$ is an induced matching of $H_k$ and there is no induced matching 
  which properly contains $\mathcal{M}$. 
  Suppose that $\{ u, x_1 \} \in \mathcal{M}$. 
  If $\{ v, y_j \} \in \mathcal{M}$ ($j = 1, 2, \ldots, k$), 
  then each of the rest edges is not 
  $3$-disjoint with at least one of $\{ u, x_1 \}, \{ v, y_j \}$. 
  Hence $\# \mathcal{M} = 2$. 
  If $\{ v, y_j \} \notin \mathcal{M}$ for all $j$, 
  then $\mathcal{M}$ contains exactly one of 
  $\{ y_{\ell}, z_{\ell 1} \}, \{ y_{\ell}, z_{\ell 2} \}$ 
  for each $\ell = 2, 3, \ldots, k$. 
  Therefore $\# \mathcal{M} = 1 + (k-1) = k$. 
  When $\{ u, x_j \}, \{ v, y_j \} \notin \mathcal{M}$ for all $j$, 
  exactly one of 
  $\{ x_j, z_{j1} \}, \{ x_j, z_{j2} \}, 
   \{ y_j, z_{j1} \}, \{ y_j, z_{j2} \}$ belongs to $\mathcal{M}$. 
  Thus $\# \mathcal{M} = k$. Therefore we have $\ind-match (H_k) = k$. 

  \par
  Let $\mathcal{M}'$ be a maximal matching of $H_k$. 
  We show that $\# \mathcal{M}' = 2k$. 
  Suppose that $\{ u, x_1 \} \in \mathcal{M}'$. 
  If $\{ v, y_1 \} \in \mathcal{M}'$, then 
  $\mathcal{M}' \setminus \{ \{ u, x_1 \}, \{ v, y_1 \} \}$ is 
  a maximal matching of $k-1$ copies of the $4$-cycle whose cardinality 
  is $2(k-1)$. Hence $\# \mathcal{M}' = 2k$. 
  If $\{ v, y_j \} \in \mathcal{M}'$ for $j \neq 1$, 
  then $\mathcal{M}' \setminus \{ \{ u, x_1 \}, \{ v, y_j \} \}$ is 
  a maximal matching of the disjoint union of $2$ copies of $P_3$ 
  (the path graph with $3$ vertices) and $k-2$ copies of the $4$-cycle. 
  Hence it follows that $\# \mathcal{M}' = 2+ 2 + 2(k-2) = 2k$. 
  If $\{ v, y_j \} \notin \mathcal{M}'$ for all $j$, 
  then $\mathcal{M}' \setminus \{ \{ u, x_1 \} \}$ is 
  a maximal matching of the disjoint union of $P_3$ and
  $k-1$ copies of the $4$-cycle. 
  Hence $\# \mathcal{M}' = 1 + 1 + 2(k-1) = 2k$. 
  When $\{ u, x_j \}, \{ v, y_j \} \notin \mathcal{M}$ for all $j$, 
  $\mathcal{M}'$ is a maximal matching of $k$ copies of the $4$-cycle. 
  Therefore $\# \mathcal{M}' = 2k$. 

  \par
  Finally, we compute $\reg (K[V_k]/I(H_k))$. 
  We consider the following decomposition of $H_k$: 
  (a) the $4$-cycle with vertices $x_j, z_{j1}, y_j, z_{j2}$ and the edge 
  $\{ x_j, u \}$ ($j=1, 2, \ldots, k$); (b) the star graph on 
  $\{ v, y_1, \ldots, y_k \}$. 
  The edge ideal of each decomposed graph is of regularity $1$. 
  Thus we have $\reg (K[V_k]/I(H_k)) \leq k+1$ by Lemma \ref{KM:regularity}. 

  \par
  For the opposite inequality $\reg (K[V_k]/I(H_k)) \geq k+1$, 
  we consider the induced subgraph of $H_k$ on 
  \begin{displaymath}
    W_k = \{ u, v, \; x_1, x_2, \ldots, x_k, \; y_1, y_2, \ldots, y_k, \; 
       z_{11}, z_{21}, \ldots, z_{k1} \}. 
  \end{displaymath}
  We use the labeling of edges of $H_k$ as in (\ref{eq:label-edgeH_k}). 
  For the sake of simplicity, we use $z_i$ 
  (resp.\  $\mathfrak{g}_{i}$, $\mathfrak{h}_{i}$) 
  instead of $z_{i1}$ 
  (resp.\  $\mathfrak{g}_{i1}$, $\mathfrak{h}_{i1}$) 
  for $i = 1, 2, \ldots, k$ and 
  denote this graph by $H_k'$; 
  see Figure \ref{fig:H_k'}. 
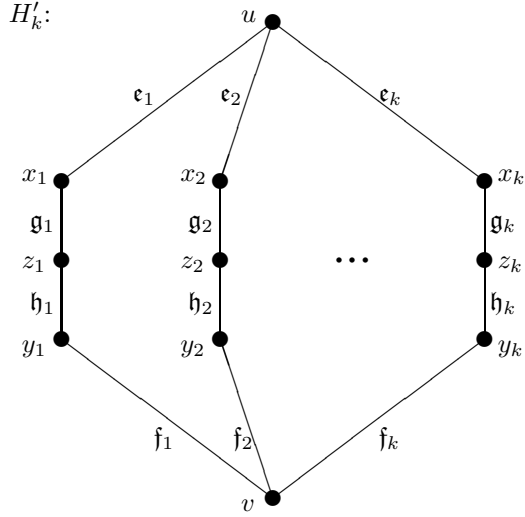
\begin{figure}[htbp]
  \begin{center}
    \begin{picture}(200,230)(-20,-105)
      \put(-20,90){$H_k'$:}
%      \thicklines
      \put(0,0){\circle*{6}}
      \put(60,0){\circle*{6}}
      \put(105,0){\circle*{2}}
      \put(110,0){\circle*{2}}
      \put(115,0){\circle*{2}}
      \put(160,0){\circle*{6}}
      \put(0,30){\circle*{6}}
      \put(60,30){\circle*{6}}
      \put(160,30){\circle*{6}}
      \put(0,-30){\circle*{6}}
      \put(60,-30){\circle*{6}}
      \put(160,-30){\circle*{6}}
      \put(80,90){\circle*{6}}
      \put(80,-90){\circle*{6}}
      %%%%%%%%%
      \put(-15,-3){$z_{1}$}
      \put(45,-3){$z_{2}$}
      \put(165,-3){$z_{k}$}
%      \put(-17,-3){$z_{11}$}
%      \put(43,-3){$z_{21}$}
%      \put(165,-3){$z_{k1}$}
      \put(-15,30){$x_{1}$}
      \put(45,30){$x_{2}$}
      \put(165,30){$x_{k}$}
      \put(-15,-35){$y_{1}$}
      \put(45,-35){$y_{2}$}
      \put(165,-35){$y_{k}$}
      \put(68,90){$u$}
      \put(68,-95){$v$}
      %%%%%%%%%
      \put(0,0){\line(0,1){30}}
      \put(60,0){\line(0,1){30}}
      \put(160,0){\line(0,1){30}}
      \put(0,0){\line(0,-1){30}}
      \put(60,0){\line(0,-1){30}}
      \put(160,0){\line(0,-1){30}}
      \put(80,90){\line(-4,-3){80}}
      \put(80,90){\line(-1,-3){20}}
      \put(80,90){\line(4,-3){80}}
      \put(80,-90){\line(-4,3){80}}
      \put(80,-90){\line(-1,3){20}}
      \put(80,-90){\line(4,3){80}}
      %%%%%%%%%
      \put(27,60){$\mathfrak{e}_1$}
      \put(60,60){$\mathfrak{e}_2$}
      \put(120,60){$\mathfrak{e}_k$}
      \put(35,-70){$\mathfrak{f}_1$}
      \put(65,-70){$\mathfrak{f}_2$}
      \put(120,-70){$\mathfrak{f}_k$}
      \put(-12,12){$\mathfrak{g}_1$}
      \put(48,12){$\mathfrak{g}_2$}
      \put(162,12){$\mathfrak{g}_k$}
      \put(-12,-19){$\mathfrak{h}_1$}
      \put(48,-19){$\mathfrak{h}_2$}
      \put(162,-19){$\mathfrak{h}_k$}
    \end{picture}
  \end{center}
  \caption{The graph $H_k'$}
  \label{fig:H_k'}
\end{figure}
%
%  \par
  By Hochster's formula for Betti numbers (see also \cite[Lemma 3.1]{Kimura}), 
  it is enough to prove that 
  $\beta_{2k+1, 3k+2} (K[W_k]/I(H_k')) \neq 0$. 
  In order to prove this, we use a Lyubeznik resolution. 
  We identify edges of $H_k'$ and minimal monomial generators of $I(H_k')$. 

  \par
  When $k=2$, $H_2'$ is the $8$-cycle. 
  (Hence we know that $\reg K[W_2]/I(H_2') = 3$.) 
  Let us consider the Lyubeznik resolution of $I(H_2')$ with respect to 
  the following order of edges of $H_2'$ 
  (which corresponds to the order of minimal monomial generators 
  of $I(H_2')$): 
  \begin{equation}
    \label{eq:orderH_2'}
    \mathfrak{e}_1, \mathfrak{h}_1, \mathfrak{f}_k, \mathfrak{g}_k, 
    \mathfrak{e}_k, \mathfrak{g}_1, \mathfrak{f}_1, \mathfrak{h}_k. 
%    u x_1, z_{11} y_1, v y_2, x_2 z_{21}, u x_2, x_1 z_{11}, v y_1, z_{k2} y_2. 
%    1\mathfrak{e}_1, 2\mathfrak{h}_1, 3\mathfrak{f}_k, 4\mathfrak{g}_k, 
%    5\mathfrak{e}_k, 6\mathfrak{g}_1, 7\mathfrak{f}_1, 8\mathfrak{h}_k. 
  \end{equation}
  We denote the resolution by 
  $(\mathcal{L}_{\bullet}^{(2)}, d_{\bullet}^{(2)})$. 
  Then the maximal $L$-admissible symbols are 
  \begin{displaymath}
    [\mathfrak{e}_1, \mathfrak{h}_1, \mathfrak{f}_k, \mathfrak{g}_k, 
     \mathfrak{e}_k, \mathfrak{f}_1],  
    [\mathfrak{e}_1, \mathfrak{h}_1, \mathfrak{f}_k, \mathfrak{g}_k, 
     \mathfrak{g}_1, \mathfrak{h}_k]. 
  \end{displaymath}
  Put
  \begin{displaymath}
    \xi^{(2)} := [\mathfrak{e}_1, \mathfrak{h}_1, \mathfrak{f}_k, \mathfrak{g}_k, 
     \mathfrak{e}_k] 
     - [\mathfrak{e}_1, \mathfrak{h}_1, \mathfrak{f}_k, \mathfrak{g}_k, 
     \mathfrak{g}_1] \in \mathcal{L}_5^{(2)}. 
  \end{displaymath}
  Then it is easy to see that 
  $1 \otimes \xi^{(2)} 
    \in \Ker (1 \otimes d_{5}^{(2)}) \setminus \Image (1 \otimes d_6^{(2)})$. 
  Also $\deg \xi^{(2)} = 8$. Therefore we have 
  $\beta_{5,8} (K[V_2]/I(H_2')) \neq 0$. 

  \par
  Next assume that $k \geq 3$. 
  For each $i = 2, \ldots, k-1$, consider 
  the induced subgraph of $H_k'$ on $\{ u, x_i, z_i, y_i, v \}$. 
  We denote it by $L_i^{(k)}$. 
  Also let $C_8^{(k)}$ be the induced subgraph of $H_k'$ on 
  $\{ u, x_1, z_1, y_1, v, y_k, z_k, x_k \}$, which is the $8$-cycle. 
  Note that $E(H_k') = \bigcup_{i=2}^{k-1} E(L_i^{(k)}) \cup E(C_8^{(k)})$. 

  \par
  Let us consider the Lyubeznik resolution of $I(L_i^{(k)})$ with respect to 
  the following order of edges of $L_i^{(k)}$: 
  $\mathfrak{g}_i, \mathfrak{h}_i, \mathfrak{f}_i, \mathfrak{e}_i$. 
  We denote the resolution by 
  $(\mathcal{L}_{\bullet}^{(k,i)}, d_{\bullet}^{(k,i)})$. 
  Then the maximal $L$-admissible symbols are 
  $[\mathfrak{g}_i, \mathfrak{f}_i, \mathfrak{e}_i], 
   [\mathfrak{g}_i, \mathfrak{h}_i, \mathfrak{f}_i]$ 
  and it is easy to see that $1 \otimes [\mathfrak{g}_i, \mathfrak{h}_i] 
  \in \Ker (1 \otimes d_2^{(k,i)}) \setminus \Image (1 \otimes d_3^{(k,i)})$. 
  Also $\deg [\mathfrak{g}_i, \mathfrak{h}_i] = 3$. 

  \par
  We also consider the Lyubeznik resolution of $I(C_8^{(k)})$ with respect to 
  the ordering as in (\ref{eq:orderH_2'}). Then the same argument with 
  $H_2'$ is valid. 

  \par
  Now let us consider the Lyubeznik resolution of $I(H_k')$ with respect to 
  the following order of edges of $H_k'$: 
  \begin{equation}
    \label{eq:orderH_k'}
    \begin{aligned}
      &\mathfrak{g}_2, \mathfrak{h}_2, \mathfrak{f}_2, \mathfrak{e}_2, \\
      &\qquad \quad \ldots, \\
      &\qquad \mathfrak{g}_{k-1}, \mathfrak{h}_{k-1}, \mathfrak{f}_{k-1}, 
               \mathfrak{e}_{k-1}, \\
      &\qquad \qquad 
         \mathfrak{e}_1, \mathfrak{h}_1, \mathfrak{f}_k, \mathfrak{g}_k, 
         \mathfrak{e}_k, \mathfrak{g}_1, \mathfrak{f}_1, \mathfrak{h}_k. 
    \end{aligned}
    %    u x_1, z_{11} y_1, v y_2, x_2 z_{21}, u x_2, x_1 z_{11}, v y_1, z_{k2} y_2. 
%    1\mathfrak{e}_1, 2\mathfrak{h}_1, 3\mathfrak{f}_k, 4\mathfrak{g}_k, 
%    5\mathfrak{e}_k, 6\mathfrak{g}_1, 7\mathfrak{f}_1, 8\mathfrak{h}_k. 
  \end{equation}
  We denote the resolution by 
  $(\mathcal{L}_{\bullet}^{(k)}, d_{\bullet}^{(k)})$. 
  Note that for $i=1, \ldots, k-2$, the $i$th row of (\ref{eq:orderH_k'}) 
  corresponds to $L_{i+1}^{(k)}$ and the last row of (\ref{eq:orderH_k'}) 
  corresponds to $C_8^{(k)}$. 
  These graphs are only connected by the vertices $u,v$. 
  By the definition of the ordering of the minimal monomial generators of 
  $I(H_k')$, it is easy to see that the $L$-admissible symbols of 
  $(\mathcal{L}_{\bullet}^{(k)}, d_{\bullet}^{(k)})$ 
  are obtained by each $L$-admissible symbols 
  for $L_2^{(k)}, \ldots, L_{k-1}^{(k)}$ and $C_8^{(k)}$. 
  A similar claim is true when we consider the maximal $L$-admissible symbols. 
  Put
  \begin{displaymath}
    \begin{aligned}
      \xi^{(k)} 
      := &[\mathfrak{g}_2, \mathfrak{h}_2, 
          \ldots, \mathfrak{g}_{k-1}, \mathfrak{h}_{k-1}, \, 
          \mathfrak{e}_1, \mathfrak{h}_1, \mathfrak{f}_k, \mathfrak{g}_k, 
          \mathfrak{e}_k] \\
       - &[\mathfrak{g}_2, \mathfrak{h}_2, 
           \ldots, \mathfrak{g}_{k-1}, \mathfrak{h}_{k-1}, \, 
           \mathfrak{e}_1, \mathfrak{h}_1, \mathfrak{f}_k, \mathfrak{g}_k, 
           \mathfrak{g}_1]. 
    \end{aligned}
  \end{displaymath}
  Then $\xi^{(k)} \in (\mathcal{L}_{2(k-2)+5}^{(k)})_{3(k-2)+8}$. 
  Also 
  $1 \otimes \xi^{(k)} 
    \in \Ker (1 \otimes d_{2(k-2)+5}^{(k)}) \setminus 
        \Image (1 \otimes d_{2(k-2)+6}^{(k)})$ follows. 
  Therefore we have 
  $\beta_{2k+1,3k+2} (K[V_k]/I(H_k')) \neq 0$ as desired. 
\end{proof}

\begin{question}
  Can we construct an infinite family
  of finite simple connected graphs $G$ 
  satisfying (v)? 
\end{question}

A finite simple connected graph $G$ satisfying the inequalities (v) 
%$\match (G) = \min-match (G)$ and $\ind-match (G) < \reg S/I(G)$ 
might be rare. 
Actually, when the number of vertices of $G$ is at most $7$, 
there is no such a graph $G$ with (v) except for $C_5$. 
\begin{proposition}
  \label{7vertex}
  Let $G$ be a finite simple connected graph with at most $7$ vertices. 
  Then $\match (G) = \reg S/I(G) > \ind-match (G)$ if and only if 
  $G$ is a $5$-cycle. 
\end{proposition}

\begin{remark}
  After submitting the paper, Biyiko\u{g}lu and Civan 
  \cite[Theorem 3.17]{BC1503} proved that there is no finite simple 
  connected graph $G$ satisfying (v) except for $C_5$. 
\end{remark}

%%%%%%%%%%%%%%%%%%%%%%%%%%%%%%%%%%%%%
% \section
\section{A graph with a dominating induced matching}
\label{sec:DIM}
In \cite{HHKO}, the authors studied the Cameron--Walker graphs. 
In this section, we treat some classes of graphs which contain 
Cameron--Walker graphs as a subclass and investigate these combinatorial 
properties. 

\par
We first recall some definitions on graphs. 

\par
Let $G$ be a finite simple graph on the vertex set $V$. 
Let $W$ be a subset of $V$. We denote by $G_W$ the induced subgraph of $G$ 
on $W$: the vertex set of $G_W$ is $W$ and the edge set of $G_W$ consists of 
all edges of $G$ which are contained in $W$. 
We write $G \setminus W$ instead of $G_{V \setminus W}$. 
In particular, when $W = \{ x \}$, consisting of $1$ vertex, 
we write $G \setminus x$ instead of $G \setminus \{ x \}$. 
For a vertex $x \in V$, we denote by $N_G (x)$ the set of neighbours of $x$. 
Also we set $N_G [x] := N_G (x) \cup \{ x \}$. 
The degree of $x$ is defined by $\deg_G (x) := \# N_G (x)$. 
For a subset $W \subset V$, we set 
$N_G (W) = \bigcup_{x \in W} N_G (x)$ and 
$N_G [W] = \bigcup_{x \in W} N_G [x]$. 
We sometimes omit the lower subscript $G$ on these notation 
if there is no fear of confusion. 

\par
A subset $W \subset V$ is called \textit{independent} if 
no two vertices of $W$ are adjacent in $G$. 
An independent set $W$ is said to be maximal 
if there is no independent set of $G$ which properly contains $W$. 
Also a subset $C \subset V$ is called a \textit{vertex cover} of $G$ 
if all edges of $G$ meet with $C$. 
A vertex cover $C$ is said to be minimal if there is no vertex cover of $G$ 
which is properly contained in $C$. 
Note that $C$ is a minimal vertex cover of $G$ if and only if 
$V \setminus C$ is a maximal independent set of $G$. 
A graph is said to be unmixed if all minimal vertex covers 
(equivalently, all maximal independent sets) of $G$ have the 
same cardinality. When $G$ is unmixed, the edge ideal $I(G)$ is height unmixed. 

\par

An edge of $G$ is called a \textit{leaf edge} 
if it contains a degree $1$ vertex. 
Also a triangle of $G$ is called a \textit{pendant triangle} if its two 
vertices are of degree $2$ and the rest vertex is of degree more than $2$.

\par
\bigskip

\par
A Cameron--Walker graph $G$ satisfies the equalities (i) 
in the previous section: 
\begin{equation}
  \label{eq:CWgraph}
  \ind-match (G) = \reg S/I(G) = \min-match (G) = \match (G). 
\end{equation}
Recall that 
a Cameron--Walker graph $G$ consists of a connected bipartite graph 
with the vertex partition $X \sqcup Y$ such that there is 
at least one leaf edge attached to each vertex $x_i \in X$ 
and that there may be possibly some pendant triangles attached to a 
vertex $y_j \in Y$. 
Choose one leaf edge which contains $x_i$ for each $x_i \in X$. 
%Let $\mathcal{L}$ be the set of these edges. 
%Also we denote by $\mathcal{T}$, the set of edges those are  is 
%and one edge which contains $y_j$ among pendant triangles if 
%there are pendant triangles attached to $y_j$ for each $y_j \in Y$. 
%Then the set of these edges form an induced matching of $G$. 
%Then $\mathcal{L} \cup \mathcal{T}$ forms an induced matching of $G$. 
Then these edges and the edges consisting of two degree $2$ vertices 
of all pendant triangles form an induced matching of $G$. 
It also forms a maximal matching of $G$. 
Thus for a Cameron--Walker graph $G$, there exists an induced matching of $G$ 
which is also a maximal matching of $G$. 
Such a matching is called a \textit{dominating induced matching} or 
an \textit{efficient edge domination set}. 
There exists a graph which does not have a dominating induced matching. 
For example, let $G_0$ be the graph on the vertex set $\{ 1,2, \ldots, 6 \}$ 
with edges $\{ 1,2 \}, \{ 2,3 \}, \{ 3,4 \}, \{ 4,5 \}, \{ 3,6 \}$. 
Then it is easy to see that an induced matching consisting of 
one edge is not a maximal matching of $G_0$. 
Other induced matching of $G_0$ is only $\{ \{ 1,2 \}, \{ 4,5 \} \}$, 
which is also not a maximal matching; 
$\{ \{ 1,2 \}, \{ 4,5 \}, \{ 3,6 \} \}$ is a matching of $G_0$. 
On graph theory, it has been studied the problem 
of determining whether a given finite simple graph 
has a dominating induced matching. 
This problem is known to be NP-complete in general; 
see e.g., \cite{CMMR, LMS}. 

\par
A graph with a dominating induced matching is characterized as follows. 
It is easy to check but we give a proof of this for the completeness. 
%A subset $W \subset V$ is called an independent set if 
%any two vertices of $W$ are not adjacent in $G$. 
\begin{proposition}[{cf. \cite[p.2]{LMS}}]
   \label{ind-match=max-match}
  Let $G$ be a finite simple graph on $V$. 
%  Then there exists an induced matching of $G$ which is also a 
%  maximal matching if and only if there is an independent set $W$ 
  Then $G$ has a dominating induced matching 
  if and only if there is an independent set $W$ 
  such that $G \setminus W$ is a disjoint union of edges. 
  When this is the case, the set of edges of $G \setminus W$ forms 
  a dominating induced matching of $G$. 
\end{proposition}
%Cameron--Walker graphs are such graphs. 
\begin{proof}
%  Let $\mathcal{M} = \{ e_1, e_2, \ldots, e_s \}$ 
%  be an induced matching of $G$ 
%  which is also a maximal matching of $G$. 
  Let $\mathcal{M} = \{ e_1, e_2, \ldots, e_s \}$ be 
  a dominating induced matching of $G$. 
  Let $W$ be the set of vertices which do not appear in each $e_i$. 
  Then $W$ is an independent set of $G$ since $\mathcal{M}$ is 
  a maximal matching. 
\end{proof}
For example, the path graph $P_{6n}$ with $6n$ vertices and the $6n$-cycle 
$C_{6n}$ have dominating induced matchings; 
for instance, take $W = \{ 3, 6, \ldots, 6n \}$. 

\par
Let $G$ be a finite simple graph. 
Since $\ind-match (G) \leq \min-match (G)$ holds in general, 
if $G$ has a dominating induced matching, then
$\ind-match (G) = \min-match (G)$ holds. 
\begin{remark}
  \label{rmk:ind-match<min-match}
  For a finite simple graph $G$, 
  the inequality $\ind-match (G) \leq \min-match (G)$ follows via 
  $\reg S/I(G)$; see (\ref{eq:all-ineq}). 

  \par
  We prove this inequality by pure combinatorics. 
  Let $\mathcal{M} = \{ e_1, \ldots, e_s \}$ be an induced matching of $G$
  and $\mathcal{M}' = \{ e_1', \ldots, e_t' \}$ a maximal matching of $G$. 
  Then for each $e_k \in \mathcal{M}$, 
  there exists $e_{i_k}' \in \mathcal{M}'$ with 
  $e_{i_k}' \cap e_{k} \neq \emptyset$ 
  because of the maximality of $\mathcal{M}'$ . 
  Since $\mathcal{M}$ is an induced matching, it follows that $i_k \neq i_j$ 
  if $k \neq j$. 
  Therefore we have $\ind-match (G) \leq \min-match (G)$. 
\end{remark}

\par
We next characterize a finite simple graph $G$ 
with $\ind-match (G) = \min-match (G)$. 
\begin{theorem}
  \label{ind-match=min-match}
  Let $G$ be a finite simple graph on $V$. 
  Then $G$ satisfies $\ind-match (G) = \min-match (G)$ if and only if 
  the vertex set $V$ can be partitioned as 
  \begin{displaymath}
%    \begin{aligned}
      V = \{ v_{i1}, v_{i2} \; : \; i = 1, 2, \ldots, \alpha + \beta \} 
        \sqcup \{ z_1, \ldots, z_{\alpha} \} 
        \sqcup \{ w_1, \ldots, w_{\gamma} \}, 
%      V &= \{ v_{i1}, v_{i2} \; : \; i = 1, 2, \ldots, \alpha + \beta \} \\
%        &\sqcup \{ z_1, \ldots, z_{\alpha} \} \\
%       &\sqcup \{ w_1, \ldots, w_{\gamma} \}, 
%    \end{aligned}
  \end{displaymath}
  where $\alpha, \beta, \gamma$ are non-negative integers, 
  so that the edge set of $G$ is of the following form 
  \begin{displaymath}
%    \begin{aligned}
      \{ e_i \; : \; i = 1, 2, \ldots, \alpha + \beta \} \\
      \cup \{ e_i' \; : \; i = 1, 2, \ldots, \alpha \} \cup E', 
%      &\{ e_i \; : \; i = 1, 2, \ldots, \alpha + \beta \} \\
%      &\cup \{ e_i' \; : \; i = 1, 2, \ldots, \alpha \} \cup E', 
%    \end{aligned}
  \end{displaymath}
  where we set 
  $e_i = \{ v_{i1}, v_{i2} \}$ ($i = 1, 2, \ldots, \alpha + \beta$) and 
  $e_i' = \{ v_{i1}, z_i \}$ ($i = 1, 2, \ldots, \alpha$), 
  and an edge in $E'$ is one of the following: 
  \begin{enumerate}
  \item[(i)] an edge containing $z_i$ ($i = 1, 2, \ldots, \alpha$); 
  \item[(ii)] an edge consisting of an end vertex of $e_i$ 
    ($i = \alpha + 1, \alpha + 2, \ldots, \alpha + \beta$) 
    and $w_j$ ($j = 1, 2, \ldots, \gamma$); 
  \item[(iii)] an edge consisting of $v_{i1}$ ($i = 1, 2, \ldots, \alpha$) 
    and $w_j$ ($j = 1, 2, \ldots, \gamma$). 
  \end{enumerate}
%  In other words, $\{ w_1, \ldots, w_{\gamma} \}$ is an idependent set of $G$ 
%  and $N (v_{i2}) \subset \{ z_1, \ldots, z_{\alpha}, v_{i2} \}$ for 
%  $i = 1, 2, \ldots, \alpha$. 
\end{theorem}
\begin{proof}
  {\bf (``If'')} 
  It is easy to see that $e_1, \ldots, e_{\alpha + \beta}$ 
  form an induced matching of $G$. 
  Hence we have $\alpha + \beta \leq \ind-match (G)$. 
  On the other hand, 
  $e_1', \ldots, e_{\alpha}', e_{\alpha + 1}, \ldots, e_{\alpha + \beta}$ 
  form a maximal matching of $G$ since the rest vertices are 
  $v_{12}, \ldots, v_{\alpha 2}, w_1, \ldots, w_{\gamma}$. 
  Hence we have $\min-match (G) \leq \alpha + \beta$. 
  By combining these inequalities, 
  we have $\ind-match (G) = \min-match (G) = \alpha + \beta$. 

  \par
  {\bf (``Only If'')} 
  Put $s = \ind-match (G) = \min-match (G)$. 
  Let $\mathcal{M} = \{ e_1, \ldots, e_s \}$ 
  be an induced matching of $G$ with cardinality $s$ 
  and $\mathcal{M}' = \{ e_1', \ldots, e_s' \}$ 
  a maximal matching of $G$ with cardinality $s$. 
  As noted in Remark \ref{rmk:ind-match<min-match}, 
  for each $e_k \in \mathcal{M}$, 
  there exists an edge $e_{i_k}' \in \mathcal{M}'$ with 
  $e_{i_k}' \cap e_k \neq \emptyset$, and $i_k \neq i_j$ if $k \neq j$. 
  Also since $\# \mathcal{M} = s = \ind-match (G)$, 
  $i_k$ is uniquely determined by $k$. Therefore we may assume that 
  $e_k' \neq e_k$ for $k = 1, \ldots, \alpha$ ($0 \leq \alpha \leq s$) and 
  $e_k' = e_k$ for $k = \alpha + 1, \ldots, s$. 
  Set $z_k = e_k' \setminus e_k$ ($k = 1, \ldots, \alpha$) and 
  $W = V \setminus (\bigcup_{k=1}^s e_k \cup \{ z_1, \ldots, z_{\alpha} \})$. 
  Then it is easy to see that 
  \begin{displaymath}
    V = \left( \bigcup_{k=1}^s e_k \right) 
          \sqcup \{ z_1, \ldots, z_{\alpha} \} \sqcup W 
  \end{displaymath}
  is a desired partition. 
\end{proof}

\par
As mentioned before Remark \ref{rmk:ind-match<min-match}, 
if $G$ has a dominating induced matching, 
then $\ind-match (G) = \min-match (G)$ holds. 
But the converse is false; 
the graph $G_0$ (see the beginning of this section) 
does not have a dominating induced matching, 
but $\ind-match (G_0) = \min-match (G_0) = 2$; 
for example, $\{ \{ 2, 3 \}, \{ 4, 5 \} \}$ is a maximal matching with 
cardinality $2$. 

\par
\bigskip

\par
Now we return to the graph with a dominating induced matching. 
We consider the problem which graph with a dominating induced matching 
is unmixed. 
Let $G$ be a finite simple graph on the vertex set $V$ 
with a dominating induced matching. 
Then $V$ can be decomposed as $W \sqcup M$ where 
$W$ is an independent set of $G$ and 
$G_M$ consists of $m := \min-match (G)$ disconnected edges 
$\{ x_{j1}, x_{j2} \}$, $j=1, \ldots, m$. 
Set 
\begin{displaymath}
  W_0 := \{ w \in W \; : \; \text{$w$ is not an isolated vertex of $G$} \}. 
\end{displaymath}
Also set 
\begin{displaymath}
  \begin{aligned}
    m_1 &:= \# \{ j \; : \; 
      \text{$\deg_G x_{j1} = 1$ or $\deg_G x_{j2} = 1$} \}, \\
    m_2 &:= \# \{ j \; : \; 
      \text{$\deg_G x_{j1} \geq 2$ and $\deg_G x_{j2} \geq 2$} \}. 
  \end{aligned}
\end{displaymath}
Note that $m_1 + m_2 = m$. 
For a subset $U \subset V$, we 
%put 
%\begin{displaymath}
%  N_G (V_0) := \{ v \in V \; : \; 
%                 \text{$\{ v,v_0 \} \in E(G)$ for some $v_0 \in V_0$} \} 
%\end{displaymath}
%and 
denote by $IN_G (U, W)$ (or $IN (U, W)$ if there is no fear of confusion), 
the set of isolated vertices of $G \setminus N_G [U]$ 
which are contained in $W_0$. 
%If there is no fear of confusion, we denote $N_G (V_0)$, $IN_G (V_0, W)$ 
%by $N (V_0)$, $N (V_0, W)$, respectively, for the simplicity. 

\par
Let $M_2$ be a subset of $M$ satisfying the following $2$ conditions: 
\begin{enumerate}
\item[$(\ast 1)$] $\# (M_2 \cap \{ x_{j1}, x_{j2} \}) \leq 1$ 
  for all $j = 1, \ldots, m$; 
\item[$(\ast 2)$] $\deg_G x \geq 2$ for all $x \in M_2$, 
\end{enumerate}
Then $G' := G \setminus N_G [M_2]$ has a dominating induced matching 
if $G'$ is not an edgeless graph. 
Indeed $V' := V(G') = V \setminus N_G [M_2]$, 
the vertex set of $G'$, can be decomposed as 
$W' \sqcup M'$, where 
$W' = W \setminus N_G [M_2]$ and $M' = M \setminus N_G [M_2]$. 
We use notation $m_1', m_2', W_0'$ for $G'$ with respect to 
this decomposition of $V'$ 
with a similar meaning to $G$. 
If $G'$ is an edgeless graph, 
then we set $m_1' = m_2' = 0$ and $W_0' = \emptyset$. 
\begin{theorem}
  \label{claim:unmixedDIM}
  We use the same notation as above. 
  Let $G$ be a finite simple graph on $V$ with a dominating induced matching. 
  Then $G$ is unmixed if and only if for some (and then all) 
  decomposition $V = W \sqcup M$, 
  the following condition $(\flat)$ is satisfied: 
  \begin{description}
    \item[Condition $(\flat)$] 
    For each subset $M_2 \subset M$ with the properties $(\ast 1)$, $(\ast 2)$, 
    the following $2$ conditions are satisfied: 
   \begin{displaymath}
     \begin{aligned}
       (\flat 1) &\  m_2 - m_2' = \# N_G (M_2) - \# M_2. \\
       (\flat 2) &\  
         \# W_0 \leq 2 m_2 - \# N_G (M_2) 
         + \# M_2 + \# IN_G (M_2, W). 
     \end{aligned}
   \end{displaymath}
  \end{description}
\end{theorem}
\begin{remark}
  \label{rmk:GV}
  We use the same notation as in Theorem \ref{claim:unmixedDIM}. 
  \begin{enumerate}
  \item The empty set $M_2 = \emptyset$ is regarded as 
    satisfying $(\ast 1)$ and $(\ast 2)$. Then $(\flat 1)$ is satisfied 
    as both-hand sides are $0$. Also $(\flat 2)$ must be 
    $\# W_0 \leq 2 m_2$. 
    Indeed if $G$ is unmixed, this inequality holds; 
    see Lemma \ref{claim:GV} below. 
  \item The left-hand side of $(\flat 1)$ is equal to the cardinality of 
    the following set: 
    \begin{displaymath}
      \mathcal{I}_{G, M_2} := \left\{ j \; : \; 
        \begin{aligned} 
          &\deg_G x_{j1} \geq 2, \  \deg_G x_{j2} \geq 2, 
%          &\deg_G x_{j1} \geq 2 , \  
           \  \text{and  one of the following is satisfied:} \\
          &(i) \  M_2 \cap \{ x_{j1}, x_{j2} \} \neq \emptyset; \\
          &(ii) \  N_G (x_{j1}) \setminus \{ x_{j2} \} \subset N_G (M_2); \\
          &(iii) \  N_G (x_{j2}) \setminus \{ x_{j1} \} \subset N_G (M_2) 
%          &\bullet \  M_1' \cap \{ x_{j1}, x_{j2} \} \neq \emptyset; \\
%          &\bullet \  N_G (x_{j1}) \setminus \{ x_{j2} \} \subset N_G (M_1'); \\
%          &\bullet \  N_G (x_{j2}) \setminus \{ x_{j1} \} \subset N_G (M_1'). 
        \end{aligned}
        \right\}.  
    \end{displaymath}
    Indeed, $m_2 - m_2'$ counts the number of edges 
    $\{ x_{j1}, x_{j2} \}$ 
    with $\deg_G x_{j1} \geq 2$ and $\deg_G x_{j2} \geq 2$ such that 
    both $x_{j1}$ and $x_{j2}$ are not vertices of $G'$ or 
    one of $\deg_{G'} x_{j1} = 1$ and $\deg_{G'} x_{j2} = 1$ holds. 
  \end{enumerate}
\end{remark}

\par
We first prove the following lemma. 
\begin{lemma}
  \label{claim:GV}
  We use the same notation as above. 
  Let $G$ be a finite simple graph on $V$ with a dominating induced matching: 
  $V = W \sqcup M$. 
  If $G$ is unmixed, then $\# W_0 \leq 2 m_2$. 
\end{lemma}
\begin{proof}
  Let us consider the following subset of $V$: 
  \begin{equation}
    \label{eq:MinimalVertexCover}
    C_0 := \left( \bigcup_{\genfrac{}{}{0pt}{}{j}{\deg_G x_{j1} \geq 2, \, \deg_G x_{j2} \geq 2}} 
      \{ x_{j1}, x_{j2} \} \right) 
    \cup \left( \bigcup_{\genfrac{}{}{0pt}{}{j}{\deg_G x_{j k} \geq \deg_G x_{j \ell} = 1}} 
      \{ x_{j k} \} \right). 
  \end{equation}
  Then $C_0$ is a minimal vertex cover of $G$. Since 
  $G$ is unmixed, we have $\height I(G) = m_1 + 2 m_2$. 
  By Gitler and Valencia \cite[Corollary 3.4]{GV}, 
  we have $2 \height I(G) \geq \# W_0 + \# M$. 
  Note that $\# M = 2 (m_1 + m_2)$. 
  Hence 
  \begin{displaymath}
    \# W_0 \leq 2 \height I(G) - \# M 
           = 2 (m_1 + 2 m_2) - 2 (m_1 + m_2) 
           = 2 m_2. 
  \end{displaymath}
\end{proof}

Now we prove Theorem \ref{claim:unmixedDIM}. 
\begin{proof}[Proof of Theorem \ref{claim:unmixedDIM}]
  We first assume that $G$ is unmixed. 
  Let $W \sqcup M$ be a decomposition of $V$ where 
  $W$ is an independent set of $G$ and $G_M$ consists of 
  disconnected edges $\{ x_{j1}, x_{j2} \}$, $j=1, \ldots, m$. 
  Since $C_0$ in (\ref{eq:MinimalVertexCover}) 
  is a minimal vertex cover of $G$, 
  the cardinality of any minimal vertex cover of $G$ is 
  $m_1 + 2 m_2$. 
  Let $M_2$ be a subset of $M$ satisfying the conditions 
  $(\ast 1)$ and $(\ast 2)$. 
  Then $G' = G \setminus N_G [M_2]$ is an edgeless graph or a graph with 
  dominating induced matching as noted before Theorem 
  \ref{claim:unmixedDIM}. 
  Considering the minimal vertex cover of $G$ which is disjoint with $M_2$, 
  we have $m_1 + 2 m_2 = \# N_G (M_2) + m_1' + 2 m_ 2'$. 
  (If $G'$ is an edgeless graph, 
  then we consider both $m_1'$ and $m_2'$ as $0$.) 
  Also focusing on the number of edges of $G_M$, we have 
  $m_1 + m_2 = \# M_2 + m_1' + m_2'$. 
  Then we have 
  \begin{equation}
    \label{eq:m2-m2'}
    m_2 - m_2' = \# N_G (M_2) - \# M_2. 
  \end{equation}
  Hence $(\flat 1)$ holds. 
  Also note that when $G$ is unmixed, $G'$ is also unmixed 
  since the union of a minimal vertex cover of $G'$ 
  and $N_G (M_2)$ is a minimal vertex cover of $G$. 
  Then by Lemma \ref{claim:GV}, we have 
  $\# W_0' \leq 2 m_2'$. 
  Since 
  \begin{displaymath}
      \# W_0' = \# W_0 - (\# N_G (M_2) - \# M_2) - \# IN_G (M_2, W) 
  \end{displaymath}
  and (\ref{eq:m2-m2'}), we have 
  \begin{displaymath}
    \begin{aligned}
      \# W_0 &= \# W_0' + (\# N_G (M_2) - \# M_2) + \# IN_G (M_2, W) \\
             &\leq 2 m_2' + (\# N_G (M_2) - \# M_2) + \# IN_G (M_2, W) \\
             &= 2 (m_2 - \# N_G (M_2) + \# M_2) 
              + (\# N_G (M_2) - \# M_2) + \# IN_G (M_2, W) \\
             &= 2 m_2 - \# N_G (M_2) + \# M_2 + \# IN_G (M_2, W). 
    \end{aligned}
  \end{displaymath}
  Thus $(\flat 2)$ also holds. 

  \par
  \smallskip

  \par
%  We next assume that $G$ satisfies the condition $(\flat)$. 
  We next assume that the decomposition $V=W \sqcup M$
  satisfies the condition $(\flat)$. 
%  We may assume that $G$ has no isolated vertex. 
  As noted in Remark \ref{rmk:GV}, 
  the inequality $\# W_0 \leq 2 m_2$ is satisfied. 
  We use induction on $m$. 

  \par
  When $m = 1$, there are 2 cases: 
  $(m_1, m_2) = (1,0), (0,1)$. 

  \par
  If $(m_1, m_2) = (1,0)$, then $\# W_0 \leq 2 m_2 = 0$. 
  Therefore it follows that $G$ is a graph consisting of a single edge with 
  isolated vertices and thus $G$ is unmixed. 

  \par
  If $(m_1, m_2) = (0,1)$, 
  then $\# W_0 \leq 2 m_2 = 2$. Also, since $m_2 = 1 > 0$, 
  we have $\# W_0 > 0$. 
  Hence $\# W_0 = 1, 2$. 
  We first assume that $\# W_0 = 1$. 
  Since $\deg_G x_{11}, \deg_G x_{12} \geq 2$, it follows that 
  $G$ is a triangle with isolated vertices. Thus it is unmixed. 
  We next assume that $\# W_0 = 2$. Put $W_0 = \{ w_1, w_2 \}$. 
  Take $M_2 = \{ x_{12} \}$. Then $M_2$ satisfies the conditions 
  $(\ast 1)$ and $(\ast 2)$. 
  By $(\flat 2)$, we have 
  \begin{displaymath}
    \begin{aligned}
      2 = \# W_0 &\leq 2 m_2 - \# N_G (M_2) + \# M_2 + \# IN_G (M_2, W) \\
                 &= 2 - \# N_G (x_{12}) + 1 + \# IN_G (\{ x_{12} \}, W). 
    \end{aligned}
  \end{displaymath}
  Hence $\deg_G x_{12} = \# N_G (x_{12}) \leq \# IN_G (\{ x_{12} \}, W) + 1$. 
  Note that $\deg_G x_{12} = 2,3$. 
  If $\deg_G x_{12} = 3$, then $\{ x_{12}, w_1 \}, \{ x_{12}, w_2 \} \in E(G)$ 
  and $IN_G (\{ x_{12}\}, W) = \emptyset$. 
  This contradicts to $\deg_G x_{12} \leq \# IN_G (\{ x_{12} \}, W) +1$. 
  Hence $\deg_G x_{12} = 2$. The same is true for $x_{11}$. 
  Therefore we conclude that the edge set of $G_{W_0 \cup M}$ is, 
  by renumbering the vertices, 
  $\{ \{ x_{11}, w_1 \}, \{ x_{12}, w_2 \}, \{ x_{11}, x_{12} \} \}$, 
  and thus $G$ is unmixed. 

  \par
  We next assume that $m \geq 2$. 
  Since $C_0$ is a minimal vertex cover 
  of $G$ with cardinality $m_1 + 2 m_2$, 
  it is sufficient to prove that the cardinality of 
  any minimal vertex cover of $G$ is $m_1 + 2 m_2$. 
  Let $C$ be a minimal vertex cover of $G$ which is not of the form $C_0$. 
  Then there exists a vertex in $M$, say $x_{m 2}$, 
  with $x_{m 2} \notin C$ such that $\deg_G x_{m 2} \geq 2$. 
  Then $N_G (x_{m 2}) \subset C$. 
  As noted in Lemma \ref{claim:min-ver-cov} below, 
  we have that $C \setminus N_G (x_{m 2})$ is a minimal vertex cover of 
  $G \setminus N_G [x_{m 2}]$. 

  \par
  Put $G'' := G \setminus N_G [x_{m 2}]$. 
  Then $G''$ is also a graph with a dominating induced matching. 
  Let $W'' \sqcup M''$ be the decomposition of the vertex set $V'' = V(G'')$ 
  induced by the decomposition $V=W \sqcup M$. 
  Note that $m_1'' + m_2'' = m'' = m - 1$. 
  Then it is sufficient to prove that $V'' = W'' \sqcup M''$ 
  also satisfies the condition $(\flat)$ for $G''$. 
  Indeed, when this is the case, it follows that $G''$ is unmixed 
  by inductive hypothesis. 
  Therefore
  \begin{displaymath}
    \# (C \setminus N_G (x_{m 2})) = m_1'' + 2 m_2''. 
  \end{displaymath}
  Consider the condition $(\flat)$ with $\{ x_{m 2} \}$. 
%  As noted in the former proof (see below of (\ref{eq:m2-m2'})), 
  By $(\flat 1)$, we have 
  \begin{displaymath}
      m_2 - m_2'' = \# N_G (x_{m 2}) - 1. 
  \end{displaymath}
  Hence
  \begin{displaymath}
    \begin{aligned}
      \# C &= \# (C \setminus N_G (x_{m 2})) + \# N_G (x_{m 2}) \\
           &= (m_1'' + 2 m_2'') + (m_2 - m_2'' + 1) \\
           &= m_1'' + m_2'' + 1 + m_2 \\
           &= m + m_2 \\
           &= m_1 + 2 m_2, 
    \end{aligned}
  \end{displaymath}
  as required. 

  \par
%  Now we prove that $G''$ also satisfies the condition $(\flat)$. 
  Now we prove that $V'' = W'' \sqcup M''$ 
  also satisfies the condition $(\flat)$ for $G''$. 

  \par
%  Let $V'$ be the vertex set of $G'$. 
  Let $M_2''$ be a subset of $M''$ satisfying $(\ast 1)$ and $(\ast 2)$ 
  for $G''$. 
  We need to prove that $(\flat 1)$ and $(\flat 2)$ are satisfied. 
  In order to prove $(\flat 1)$, we use the description of the left-hand side 
  of $(\flat 1)$ as in Remark \ref{rmk:GV} (2). 
  Put $M_2 = M_2'' \cup \{ x_{m 2} \}$. 
  Note that $M_2$ and $\{ x_{m2} \}$ satisfy 
  $(\ast 1)$ and $(\ast 2)$ for $G$. 
%  $(\flat 1)$ is satisfied for the pair $(G, M_2)$. 
  Also note that the right-hand side of $(\flat 1)$ for $(G, M_2)$ is 
  \begin{equation}
    \label{eq:neighbour}
    \# N_G (M_2) - \# M_2 
    = (\# N_{G''} (M_2'') - \# M_2'') + (\# N_G (x_{m 2}) - 1) 
  \end{equation}
  because $\# M_2 = \# M_2'' + 1$ and 
  $\# N_G (M_2) = \# N_{G''} (M_2'') + \# N_G (x_{m 2})$. 

  \par
  Now, let $j$ be an index with $\deg_G x_{j1} \geq 2$ 
  and $\deg_G x_{j2} \geq 2$. 
  Recall that the left-hand side of $(\flat 1)$ for $(G, M_2)$ is the number of 
  $j$ for which one of the condition (i), (ii), (iii) 
  inside $\mathcal{I}_{G, M_2}$ is satisfied. 
  We compare the satisfaction of the condition for the pair $(G, M_2)$ 
  with that for the pair $(G'', M_2'')$. 
  If $j \neq m$, and 
  $\deg_{G''} x_{j1} \geq 2$ and $\deg_{G''} x_{j2} \geq 2$, 
  then the satisfaction of each of the conditions (i), (ii), (iii) 
  inside $\mathcal{I}_{G, M_2}$ and $\mathcal{I}_{G'', M_2''}$ is equivalent.  
  If $j=m$, then note that $x_{m 2} \in \{ x_{m2} \} \subset M_2$, 
  that is $\{ x_{m 2} \} \cap \{ x_{m1}, x_{m2} \} \neq \emptyset$ as well as 
  $M_2 \cap \{ x_{m1}, x_{m2} \} \neq \emptyset$, 
  which corresponds to the condition (i) 
  inside $\mathcal{I}_{G, \{ x_{m2} \}}$, $\mathcal{I}_{G, M_2}$, respectively. 
  If $j \neq m$, and $\deg_{G''} x_{j1} \leq 1$ or 
  $\deg_{G''} x_{j2} \leq 1$, then 
  one of the following is satisfied: 
  \begin{displaymath}
    \begin{aligned}
      N_G (x_{j1}) \setminus \{ x_{j2} \} 
        &\subset N_G (x_{m 2}) \subset N_G (M_2), \\
      N_G (x_{j2}) \setminus \{ x_{j1} \} 
        &\subset N_G (x_{m 2}) \subset N_G (M_2). 
    \end{aligned}
  \end{displaymath}
  These correspond to the conditions (ii), (iii) 
  inside $\mathcal{I}_{G, \{ x_{m2} \} }$, $\mathcal{I}_{G, M_2}$. 
  Note that when $j \neq m$, and $\deg_{G''} x_{j1} \geq 2$ and 
  $\deg_{G''} x_{j2} \geq 2$, the cases (i), (ii), (iii) inside 
  $\mathcal{I}_{G, \{ x_{m2} \} }$ do not occur.  
  Combining these with Remark \ref{rmk:GV} (2), we have that  
  the lefthand-side of ($\flat 1$) for $M_2$ with respect to $G$ 
  is equal to 
  the sum of the lefthand-side of ($\flat 1$) for $\{ x_{m 2} \}$ 
  with respect to $G$ 
  and the lefthand-side of ($\flat 1$) for $M_2''$ with respect to $G''$. 
%  Recall that  $M_2$ and $\{ x_{m2} \}$ satisfy 
%  $(\ast 1)$ and $(\ast 2)$ for $G$. 
  Hence by assumption for $G$, we have that the lefthand-side of 
  ($\flat 1$) for $M_2''$ with respect to $G''$ is equal to 
  \begin{displaymath}
    (\# N_G (M_2) - \# M_2) 
    - (\# N_G (\{ x_{m 2} \}) - \# \{ x_{m 2} \})
  \end{displaymath}
  By (\ref{eq:neighbour}), it is equal to 
  \begin{displaymath}
    \# N_{G''} (M_2'') - \# M_2'', 
  \end{displaymath}
  as desired. 

  \par
  Finally we prove the inequality ($\flat 2$) 
  for $M_2''$ with respect to $G''$. 
  Let $W_0''$ be the set of vertices in $V'' \cap W'' = V'' \cap W$ 
  which are not isolated in $G''$. 
  Then 
  \begin{displaymath}
    \# W_0'' = \# W_0 - (\# N_G (x_{m 2}) - 1) - \# IN_G (\{ x_{m 2} \}, W). 
  \end{displaymath}
  Also 
  \begin{displaymath}
    \begin{aligned}
      \# N_{G''} (M_2'') &= \# N_G (M_2) - \# N_G (x_{m2}), \\
      \# IN_{G''} (M_2'', W'') 
         &= \# IN_G (M_2, W) - \# IN_G (\{ x_{m 2} \}, W). 
    \end{aligned}
  \end{displaymath}
  Furthermore, it follows from the assumption $(\flat 1)$ 
  for $\{ x_{m 2} \}$ with respect to $G$ that 
  $m_2'' = m_2 - \# N_G (x_{m 2}) + 1$. 
  Then 
  \begin{displaymath}
    \begin{aligned}
      &2 m_2'' - \# N_{G''} (M_2'') + \# M_2'' + \# IN_{G''} (M_2'', W'') 
      - \# W_0'' \\
      &= 2 (m_2 - \# N_G (x_{m 2}) + 1) 
       - (\# N_G (M_2) - \# N_G (x_{m 2})) 
       + (\# M_2 - 1) \\
      &+ (\# IN_G (M_2, W) - \# IN_G (\{ x_{m 2} \}, W)) 
       - (\# W_0 - (\# N_G (x_{m 2}) - 1) - \# IN_G (\{ x_{m 2} \}, W)) \\
      &= 2 m_2 - \# N_G (M_2) + \# M_2 + \# IN_G (M_2, W) - \# W_0 \geq 0
    \end{aligned}
  \end{displaymath}
  by the assumption $(\flat 2)$ for $M_2$ with respect to $G$. 
  Hence $(\flat 2)$ for $M_2''$ with respect to $G''$ 
  is also satisfied as desired. 
\end{proof}

\begin{lemma}
  \label{claim:min-ver-cov}
  $C \setminus N_G (x_{m 2})$ is a minimal vertex cover of 
  $G \setminus N_G [x_{m 2}]$. 
\end{lemma}
\begin{proof}
  We first prove that $C \setminus N_G (x_{m 2})$ is a vertex cover of 
  $G \setminus N_G [x_{m 2}]$. 
  Let $e$ be an edge of $G \setminus N_G [x_{m 2}]$. 
  Then $e \cap N_G [x_{m 2}] = \emptyset$. 
  Also, since $C$ is a vertex cover of $G$, 
  it follows that $e \cap C \neq \emptyset$. 
  Combining these facts we have 
  $e \cap (C \setminus N_G (x_{m 2})) \neq \emptyset$. 

  \par
  We next prove the minimality of $C \setminus N_G (x_{m 2})$. 
  Assume that $C' \subsetneq C \setminus N_G (x_{m 2})$ is a vertex cover of 
  $G \setminus N_G [x_{m 2}]$. 
  Then $C' \cup N_G (x_{m 2}) \subsetneq C$. 
  We derive a contradiction by proving that 
  $C'' := C' \cup N_G (x_{m 2})$ is a vertex cover of $G$. 

  \par
  Let $e$ be an edge of $G$. 
  If $e \cap N_G (x_{m 2}) \neq \emptyset$, 
  then $e \cap C'' \neq \emptyset$. 
  If $e \cap N_G (x_{m 2}) = \emptyset$, then $x_{m2} \notin e$ and 
  $e$ is an edge of $G \setminus N_G [x_{m 2}]$. 
  Since $C'$ is a vertex cover of $G \setminus N_G [x_{m 2}]$, 
  it follows that $e \cap C' \neq \emptyset$. 
  Therefore $e \cap C'' \neq \emptyset$. 
\end{proof}

\par
%A graph $G$ is called \textit{chordal} 
%if any cycle of $G$ whose length is more than 
%$3$ has a chord. Also 
A graph $G$ is called \textit{forest} 
if $G$ has no cycle. 
The chordalness of a graph with a dominating induced matching 
is characterized as follows: 
\begin{theorem}
  \label{chordal-vertex-edge-decomp}
  Let $G$ be a finite simple graph on $V$ with a dominating induced matching. 
  Let 
  \begin{displaymath}
    \mathcal{M} = \{ \{ x_{j1}, x_{j2} \} \; : \; 
      i=1, 2, \ldots, m \} 
  \end{displaymath}
  be a matching of $G$ so that 
  \begin{displaymath}
    W = V \setminus \bigcup_{j=1}^m \{ x_{j1}, x_{j2} \}
  \end{displaymath}
  is an independent set of $G$. 

  \par
  Let $\widetilde{G}$ be the graph obtained 
  by identifying $x_{j1}$ and $x_{j2}$ 
  for $j=1, 2, \ldots, m$. 
  That is $\widetilde{G}$ is a graph on the vertex set 
  $\widetilde{V} := W \cup \{ x_{1}, \ldots, x_m \}$ 
%  $\widetilde{V}$ where 
%  \begin{displaymath}
%    \widetilde{V} = W \cup \{ x_{1}, \ldots, x_m \}
%  \end{displaymath}
  with the edge set 
  \begin{displaymath}
    E(\widetilde{G}) = \{ \{ w, x_j \} \; : \; 
      \text{$\{ w, x_{j1} \} \in E(G)$ or $\{ w, x_{j2} \} \in E(G)$} \}. 
  \end{displaymath}

  \par
  Then $G$ is chordal if and only if 
  $\widetilde{G}$ is a forest. 
\end{theorem}
%\begin{remark}
%  Proposition \ref{chordal-vertex-edge-decomp} can be rewrite as follows. 
%  Let $F$ be a forest on $V$. Since $F$ is a bipartite graph, 
%  $V$ can be partitioned as $V = X \sqcup Y$ where $E(F) \subset X \times Y$. 
%  Then a graph obtained by replacing each vertex in $X$ to an edge 
%  and adding edges suitably according the above change 
%  is a graph with a dominating induced matching. 
%  Then the resulting graph is chordal. 
%\end{remark}
\begin{proof}
%  First suppose that $\widetilde{G}$ is a forest. 
  We first prove that if $G$ is not chordal, 
  then $\widetilde{G}$ is not a forest, 
  in other words, $\widetilde{G}$ has a cycle. 
  Assume that $G$ has a chordless cycle $C$ of length $\ell$ with $\ell > 3$. 
  Let $\widetilde{C}$ be the subgraph of $\widetilde{G}$ 
  obtained from $C$ by the same operation 
  as we obtain $\widetilde{G}$ from $G$. 
  If there is no $j$ such that both of $x_{j1}, x_{j2}$ are vertices of $C$, 
  then $\widetilde{C}$ is also a cycle. 
%  This is a contradiction because $\widetilde{G}$ is a forest.  
  Hence $\widetilde{G}$ has a cycle. 
  If both of $x_{j1}, x_{j2}$ are vertices of $C$, then 
  these must be adjacent in $C$ because $C$ is a chordless cycle. 
  Since $\ell > 3$, the other adjacent vertices $y_{i_1}$, $y_{i_2}$ 
  of $x_{j1}, x_{j2}$ are different. 
  Note that $y_{i_1}, y_{i_2} \in W$. 
  Then $y_{i_1}, y_{i_2}, x_{j}$ are vertices of $\widetilde{C}$. 
  It then follows that $\widetilde{C}$ is a cycle of $\widetilde{G}$. 
%  a contradiction. 
%  and there is other vertex of $\widetilde{C}$. 
%  Hence $\widetilde{C}$ is a cycle of $\widetilde{G}$, a contradiction. 

  \par
  Next suppose that $G$ is chordal. 
  Assume that $\widetilde{G}$ has a cycle. 
  Let $\widetilde{C}$ be a minimal cycle of $\widetilde{G}$ 
  and let $\ell$ be the length of $\widetilde{C}$. 
  Since $\widetilde{G}$ is a bipartite graph, 
  $\ell$ must be even and thus $\ell \geq 4$. 
  Let $C$ be a cycle of $G$ corresponding to $\widetilde{C}$ 
  with the minimum length. 
%  (it is determined uniquely). 
  Then the length of $C$ is greater than or equal to $\ell \geq 4$. 
  Since $G$ is chordal, $C$ must have a chord $e$. 
  We may assume that $e = \{ w, x_{j1} \}$ where $w \in W$. 
  Since $C$ is a cycle, there are two paths from $w$ to $x_{j1}$; 
  we take with the shorter length; 
  let $y_0 = w, y_1, y_2, \ldots, y_k = x_{j1}$ be a sequence of 
  vertices of such path in $C$ where $k \geq 2$ and 
  $\{ y_i, y_{i+1} \} \in E(C)$ for $i = 0, 1, \ldots, k-1$. 
  If $k > 2$, then $\{ w, x_{j1} \}$ must be a chord of $\widetilde{C}$, 
  a contradiction. 
  If $k=2$, then $y_1 = x_{j2}$ and $C \setminus \{ y_1 \}$ is also 
  a cycle corresponding to $\widetilde{C}$. 
  This contradicts to the minimality of $C$. 
%  Note that if both of $x_{j1}, x_{j2}$ are vertices of $C$, 
%  these are adjacent 
%  and the other adjacent vertices are different. 
%  Since $G$ is chordal, $C$ has a chord. 
%  Then the minimality of the length of $C$ implies that 
%  this chord is also a chord of $\widetilde{C}$. 
%  This is a contradiction because $\widetilde{C}$ 
%  is a minimal cycle of $\widetilde{G}$. 
\end{proof}

%%%%%%%%%%%%%%%%%%%%%%%%%%%%%%%%%%%%%%%%%%%%%%%%
\section{Some algebraic properties}
%\section{Algebraic properties of the edge ideal of a graph 
%with a dominating induced matching}
In this section, we investigate 
algebraic properties of the edge ideal of a graph with 
a dominating induced matching. 

\par
In \cite{HHKO}, it is proved that a Cameron--Walker graph is 
vertex decomposable, in particular, it is sequentially 
Cohen--Macaulay. 
But there is a graph with a dominating induced matching which is not 
sequentially Cohen--Macaulay; 
the $6$-cycle is such an example; see \cite[Proposition 4.1]{FVT}. 

\par
We obtain some class of vertex decomposable graphs 
among graphs with a dominating induced matching. 

\par
Recall that a graph $G$ on $V$ is called vertex decomposable 
(see \cite[Lemma 4]{Woodroofe-vertex-decomp}) if 
$G$ is an edgeless graph or there exists $v \in V$ with the 
following $2$ properties: 
\begin{enumerate}
\item[(VD1)] $G \setminus v$ and $G \setminus N[v]$ are vertex decomposable; 
\item[(VD2)] no independent set in $G \setminus N[v]$ 
  is a maximal independent set in $G \setminus v$. 
\end{enumerate}
%where $N[v] = N(v) \cup \{ v \}$ and 
%$G \setminus v$, $G \setminus N[v]$ denote $G_{V \setminus \{ v \} }$, 
%$G_{V \setminus N[v]}$, respectively. 
We call $v \in V$ a \textit{shedding vertex} of $G$ if $v$ satisfies (VD2). 
Note that for a vertex $v \in V$, if there exists $w \in V$ 
such that $N[w] \subset N[v]$, 
then $v$ is a shedding vertex (\cite[Lemma 6]{Woodroofe-vertex-decomp}). 
\begin{theorem}
  \label{DIM-VD}
  Let $G$ be a finite simple graph on $V$ with a dominating induced matching. 
  Assume that there exists a decomposition $V = W \sqcup M$ 
  satisfying the following property, 
  where $W = \{ y_1, \ldots, y_r \}$ 
  is an independent set and $G_M$ consists of 
  $m$ disconnected edges $\{ x_{j1}, x_{j2} \}$, $j = 1, \ldots, m$. 
  Moreover assume that for each $j=1, \ldots, m$, 
  one of the following is satisfied: 
  \begin{enumerate}
  \item[(i)] $\deg_G x_{j1} = 1$ or $\deg_G x_{j2} = 1$; 
  \item[(ii)] $\deg_G x_{j1} = \deg_G x_{j2} = 2$ and there is 
    $y_{i_j} \in W$ such that $x_{j1}, x_{j2} \in N_G(y_{i_j})$; 
%    $\{ x_{j1}, y_{i_j} \}, \{ x_{j2}, y_{i_j} \} \in E(G)$; 
  \item[(iii)] $\deg_G x_{jk} = 3$ and $\deg_G x_{jl} = 2$ where 
    $\{ k, l \} = \{ 1, 2 \}$, and there is $y_{i_j} \in W$ 
    such that $N_G (y_{i_j}) = \{ x_{j1}, x_{j2} \}$; 
%    with $\deg_G y_{i_j} = 2$ such that 
%    $\{ x_{j1}, y_{i_j} \}, \{ x_{j2}, y_{i_j} \} \in E(G)$; 
  \item[(iv)] $\deg_G x_{j1} = \deg_G x_{j2} = 3$ and 
    there are distinct three vertices 
    $y_{i_{j1}}, y_{i_{j2}}, y_{i_{j3}} \in W$ such that 
    $\{ x_{j1}, y_{i_{j1}} \}, \{ x_{j2}, y_{i_{j2}} \} \in E(G)$,  
    $N_G (y_{i_{j3}}) = \{ x_{j1}, x_{j2} \}$,   
%    attached to at least one of $y_{i_{j1}}, y_{i_{j2}}$. 
%    $\{ x_{j1}, y_{i_{j1}} \}, \{ x_{j2}, y_{i_{j2}} \}, 
%     \{ x_{j1}, y_{i_{j3}} \}, \{ x_{j2}, y_{i_{j3}} \} \in E(G)$, 
%    and $\deg_G y_{i_{j3}} = 2$ 
    and there is a pendant triangle 
    attached to at least one of $y_{i_{j1}}, y_{i_{j2}}$. 
  \end{enumerate}
  Then $G$ is vertex decomposable. 
\end{theorem}
\begin{remark}
  \label{rmk:DIM-VD-pendant-trianlgle}
  The each condition (ii), (iii), (iv) of Theorem \ref{DIM-VD} 
  is concerned with the existence of a pendant triangle. 

  \par
  Indeed, the condition (ii) means that $G$ has a pendant triangle attached to 
  $y_{i_j}$; the condition (iii) means that $G$ has a pendant triangle 
  attached to $x_{jk}$; the condition (iv) is explicit. 
\end{remark}

In order to prove Theorem \ref{DIM-VD}, we use the following lemma. 
\begin{lemma}
  \label{pendant-triangle-shedding}
  Let $G$ be a finite simple graph on $V$. 
  Suppose that $G$ has a pendant triangle attached to $v \in V$. 
  Then $v$ is a shedding vertex. 
\end{lemma}
\begin{proof}
  Let $v_1, v_2$ be the two degree $2$ vertices of a pendant triangle 
  attached to $v$. Then $N[v_1] = \{ v, v_1, v_2 \} \subset N[v]$ 
  holds. Hence $v$ is a shedding vertex; see before Theorem \ref{DIM-VD}.  
\end{proof}

Now we prove Theorem \ref{DIM-VD}. 
\begin{proof}[Proof of Theorem \ref{DIM-VD}]
%  We may assume that $G$ is connected. 
  We use induction on $r = \# W$. 
  If $r=1$, then $G$ is chordal and thus $G$ is vertex decomposable 
  by Woodroofe \cite[Corollary 7]{Woodroofe-vertex-decomp}. 
%  (Francisco and Van Tuyl \cite{FVT}??) 

  \par
  Suppose that $r \geq 2$. 
  If the cases (iii) and (iv) do not occur, 
  then $G$ is a Cameron--Walker graph and thus, $G$ is vertex decomposable 
  by \cite[Theorem 3.1]{HHKO}. 

  \par
  If there is an edge $\{ x_{j1}, x_{j2} \}$ with the condition (iii), 
  say, $\deg x_{j1} = 3$ and $\deg x_{j2} = 2$, 
  then $x_{j1}$ is a shedding vertex because of 
  Remark \ref{rmk:DIM-VD-pendant-trianlgle} 
  and Lemma \ref{pendant-triangle-shedding}. 
  Therefore we only need to prove that both $G \setminus x_{j1}$ and 
  $G \setminus N[x_{j1}]$ are vertex decomposable. 
  Indeed $G \setminus x_{j1}$ is the disjoint union of 
  single edge $\{ x_{j2}, y_{i_j} \}$ and 
  $G' := G \setminus \{ x_{j1}, x_{j2}, y_{i_j} \}$. 
  Since the vertex set of $G'$ can be decomposed as 
  $W' \sqcup M'$ where $W' = W \setminus \{ y_{i_j} \}$ 
  and $M' = M \setminus \{ x_{j1}, x_{j2} \}$, 
  $G'$ has a dominating induced matching. 
  Also $G'$ satisfies the assumption of the theorem 
  with this decomposition of the vertex set since 
  $N_G (y_{i_j}) = \{ x_{j1}, x_{j2} \}$. 
  Hence we conclude that $G'$, and thus $G \setminus x_{j1}$ 
  is vertex decomposable by inductive hypothesis. 
  Also the vertex set of $G \setminus N[x_{j1}]$ can be decomposed as 
  $W'' \sqcup M''$ where $W'' = W \setminus N(x_{j1})$ and 
  $M'' = M \setminus \{ x_{j1}, x_{j2} \}$. 
  Thus $G \setminus N[x_{j1}]$ has a dominating induced matching. 
  Since $\deg_G x_{j' k} \geq \deg_{G \setminus N[x_{j1}]} x_{j'k}$, 
%  and the conditions (i), (ii), (iii), (iv) are being weaker according 
%  to small number, 
  we can easily see that this decomposition satisfies the assumption of 
  the theorem. Hence by inductive hypothesis, we conclude that 
  $G \setminus N[x_{j1}]$ is also vertex decomposable. 

  \par
  Suppose that there is an edge $\{ x_{j1}, x_{j2} \}$ with the condition (iv). 
  We may assume that $G$ has a pendant triangle attached to $y_{i_{j1}}$. 
  Then $y_{i_{j1}}$ is a shedding vertex 
  by Lemma \ref{pendant-triangle-shedding}. 
  Hence it is enough to prove that both $G \setminus y_{i_{j1}}$ and 
  $G \setminus N [y_{i_{j1}}]$ are vertex decomposable. 
  We first consider $G \setminus y_{i_{j1}}$. 
  Since the vertex set of this graph can 
  be decomposed as $(W \setminus \{ y_{i_{j1}} \}) \sqcup M$, 
  this graph has a dominating induced matching. 
  We check that each $j' = 1, \ldots, m$ satisfies one of the conditions 
  (i), (ii), (iii), (iv) with respect to $G \setminus y_{i_{j1}}$. 
  If $j'$ satisfies the condition (i) (resp.\  (ii)) with respect to $G$, 
  then $j'$ satisfies the condition (i) (resp.\  (ii) or (i)) with respect to 
  $G \setminus y_{i_{j1}}$. 
  Assume that $j'$ satisfies the condition (iii) (resp.\  (iv)) 
  with respect to $G$. Since $\deg_G y_{i_{j1}} \geq 3$ and 
  $\deg_G y_{i_{j'}} = 2$ (resp.\  $\deg_G y_{i_{j' 3}} = 2$), 
  the vertex $y_{i_{j1}}$ is different from $y_{i_{j'}}$ 
  (resp.\  $y_{i_{j' 3}}$). 
  Hence $j'$ satisfies the condition (iii) or (ii) 
  (resp.\  (iv) or (iii)). 
  Therefore this decomposition satisfies the assumption of the theorem. 
  Hence by inductive hypothesis, we conclude that 
  $G \setminus y_{i_{j1}}$ is vertex decomposable. 
  We next consider $G \setminus N[y_{i_{j1}}]$. 
  In this case, the vertex set of $G \setminus N[y_{i_{j1}}]$ is decomposed as 
  $W' \sqcup M'$ where 
  \begin{displaymath}
    \begin{aligned}
      W' &= W \setminus \left( \{ y_{i_{j1}} \} \cup 
         \bigcup_{\genfrac{}{}{0pt}{}{\text{$j'$ satisfying (iii)}}{\{ x_{j' 1}, x_{j' 2} \} \cap N_G (y_{i_{j1}}) \neq \emptyset}} \{ y_{i_{j'}} \} 
         \cup 
         \bigcup_{\genfrac{}{}{0pt}{}{\text{$j'$ satisfying (iv)}}{ \# (\{ x_{j' 1}, x_{j' 2} \} \cap N_G (y_{i_{j1}})) = 1}} \{ y_{i_{j' 3}} \} \right), \\
      M' &= (M \setminus N_G (y_{i_{j1}})) \cup 
         \bigcup_{\genfrac{}{}{0pt}{}{\text{$j'$ satisfying (iii)}}{\{ x_{j' 1}, x_{j' 2} \} \cap N_G (y_{i_{j1}}) \neq \emptyset}} \{ y_{i_{j'}} \} 
         \cup 
         \bigcup_{\genfrac{}{}{0pt}{}{\text{$j'$ satisfying (iv)}}{ \# (\{ x_{j' 1}, x_{j' 2} \} \cap N_G (y_{i_{j1}})) = 1}} \{ y_{i_{j' 3}}\}. 
    \end{aligned}
  \end{displaymath}
  Then we can easily see that $G \setminus N_G [y_{i_{j1}}]$ has a dominating 
  induced matching. 
  For example, let $j'$ be an index satisfying (iii) and 
  $\{ x_{j' 1}, x_{j' 2} \} \cap N_G (y_{i_{j1}}) \neq \emptyset$, 
  say $x_{j' 1} \in N_G (y_{i_{j1}})$. Then $x_{j' 1} \notin M'$ but 
  $y_{i_{j'}} \in M$ and $\{ x_{j' 2}, y_{i_{j'}} \}$ is an edge of 
  $G \setminus N[y_{i_{j1}}]$. Note that in this case, 
  $\deg_{G \setminus N [y_{i_{j1}}]} y_{i_{j'}} = 1$. 
  Then it is also easy to see that the assumption of the theorem 
  is satisfied with this decomposition. 
  Therefore $G \setminus N[y_{i_{j1}}]$ is vertex decomposable 
  by inductive hypothesis. 
\end{proof}

\par
Although we provide the characterization for a graph with a 
dominating induced matching to be unmixed 
in Theorem \ref{claim:unmixedDIM}, 
we can obtain a clearer characterization for the unmixedness 
of the class of graphs in Theorem \ref{DIM-VD}. 
It is sufficient to consider a connected graph which is not a single edge. 

\begin{theorem}
  \label{DIM-VD-unmixed}
  Let $G$ be a finite simple connected graph on $V$ 
  with a dominating induced matching. 
  Assume that there exists a decomposition $V = W \sqcup M$ satisfying 
  the assumption of Theorem \ref{DIM-VD} (where $W \neq \emptyset$). 
  We use the same notation as in Theorem \ref{DIM-VD} 
  and before Theorem \ref{claim:unmixedDIM}. 
  Then $G$ is Cohen--Macaulay if and only if 
  $\# W = m_2$ and for all $y_i \in W$, there is just one edge 
  $\{ x_{j_i 1}, x_{j_i 2} \}$ of $G$ such that 
  both $\{ x_{j_i 1}, y_i \}$ and $\{ x_{j_i 2}, y_i \}$ are edges of $G$. 
\end{theorem}
\begin{proof}
  We first note that $G$ is Cohen--Macaulay if and only if $G$ is unmixed 
  because $G$ is vertex decomposable by Theorem \ref{DIM-VD}. 

\par
  {\bf (``Only If'')} 
  Let $M_2$ be the union of the following subsets $V_1, \ldots, V_4$ of $M$: 
  $V_1$ is the set of the vertices $x_{jk} \in V$ 
  where $\{ x_{j1}, x_{j2} \}$ is 
  an edge of type (i) of Theorem \ref{DIM-VD} 
  with $\deg x_{jk} \geq \deg x_{jl} = 1$ ($\{ k,l \} = \{ 1,2 \}$); 
  $V_2$ is the set of the vertices $x_{j1}$ where $\{ x_{j1}, x_{j2} \}$ is 
  an edge of type (ii) of Theorem \ref{DIM-VD}; 
  $V_3$ is the set of the vertices $x_{jk}$ where $\{ x_{j1}, x_{j2} \}$ is 
  an edge of type (iii) of Theorem \ref{DIM-VD} 
  with $\deg x_{jk} = 3$ and $\deg x_{jl} = 2$ ($\{ k,l \} = \{ 1,2 \}$); 
  $V_4$ is the set of the vertices $x_{jk}$ where $\{ x_{j1}, x_{j2} \}$ is 
  an edge of type (iv) of Theorem \ref{DIM-VD} 
  and the numbers of pendant triangles attached to each 
  $y_{i_j k}$ is less than or equal to that of to $y_{i_j l}$
  with the notation in Theorem \ref{DIM-VD} (iv) ($\{ k,l \} = \{ 1,2 \}$). 
  Clearly, $M_2$ satisfies the condition $(\ast 1)$. 
  The condition $(\ast 2)$ is also satisfied because $G$ is connected. 
  Note that $\# M_2 = m$, in particular $m_2' = 0$. 
  Since $G$ is unmixed, by $(\flat 1)$ of Theorem \ref{claim:unmixedDIM}, 
  it follows that 
  \begin{displaymath}
    m_2 = \# N_G (M_2) - \# M_2. 
  \end{displaymath}
  Also $W \subset N_G (M_2)$ holds. 
  Actually, take $y_i \in W$. 
  Since $G$ is connected, there exists an edge $\{ x_{j1}, x_{j2} \}$ 
  such that $\{ x_{jk}, y_i \}$ is an edge of $G$. 
  If $\{ x_{j1}, x_{j2} \}$ is of type (i) or (ii) of Theorem \ref{DIM-VD}, 
  then it is easy to see that $y_i \in N_G (M_2)$. 
  If $\{ x_{j1}, x_{j2} \}$ is of type (iii) 
  of Theorem \ref{DIM-VD} 
  and $x_{jk} \notin M_2$, then $x_{j \ell} \in M_2$ and $\deg_G x_{jk} = 2$. 
  It then follows that $y_i \in N_G (x_{j \ell}) \subset N_G (M_2)$. 
  If $\{ x_{j1}, x_{j2} \}$ is of type (iv) 
  of Theorem \ref{DIM-VD} 
  and $x_{jk} \notin M_2$, then there is a pendant triangle attached to 
  $y_i$. Let $x_{j' 1}, x_{j' 2}$ be the two vertices of the pendant triangle 
  of degree $2$. 
  Since $\{ x_{j' 1}, x_{j' 2} \}$ is of type (ii) of Theorem \ref{DIM-VD}, 
  it follows that $y_{i} \in N_G (M_2)$. 

  \par
  The inclusion $W \subset N_G (M_2)$ implies that 
  $N_G (M_2) = V(G) \setminus M_2$. Therefore 
  \begin{displaymath}
    \begin{aligned}
      m_2 &= \#N_G (M_2) - \# M_2 \\
          &= (\# V(G)- \# M_2) - \# M_2 \\
          &= \# V(G) - 2 \# M_2 \\
          &= (2m + \# W) - 2 m \\
          &= \# W. 
    \end{aligned}
  \end{displaymath}

  \par
  Suppose that there exists $y_i \in W$ such that 
  $\{ x_{j 1}, y_i \}, \{ x_{j 2}, y_i \}, 
   \{ x_{j' 1}, y_i \}, \{ x_{j' 2}, y_i \} \in E(G)$ for $j \neq j'$. 
  It then follows that 
  both $\{ x_{j1}, x_{j2} \}$ and $\{ x_{j' 1}, x_{j' 2} \}$ 
  are of type (ii) of Theorem \ref{DIM-VD}. 
  In particular, both $\{ y_i, x_{j1}, x_{j2} \}$ and 
  $\{ y_i, x_{j' 1}, x_{j' 2} \}$ form pendant triangles attached to $y_i$. 
  Assume that there are $\alpha \geq 2$ pendant triangles attached to $y_i$; 
  set the two degree $2$ vertices of each pendant triangle as 
  $\{ x_{j_k 1}, x_{j_k 2} \}$, $k=1, 2, \ldots, \alpha$. 
  Put $M_2'' = \{ x_{j_1 1}, x_{j_2 1}, \ldots, x_{j_\alpha 1} \}$. 
  Then $M_2''$ satisfies the condition $(\ast 1)$ and $(\ast 2)$. 
  Note that 
  $N_G (M_2'') = \{ y_i, x_{j_1 2}, x_{j_2 2}, \ldots, x_{j_\alpha 2} \}$ 
  and $m_2 - m_2'' = \alpha$. 
  Therefore by $(\flat 1)$ of Theorem \ref{claim:unmixedDIM}, we have 
  \begin{displaymath}
    \alpha = m_2 - m_2'' = \# N_G (M_2'') - \# M_2'' 
      = (\alpha + 1) - \alpha = 1, 
  \end{displaymath}
  this contradict to $\alpha \geq 2$. 
  Since $m_2 = \# W$, the assertion follows. 

  \par
  {\bf (``If'')} 
  Let $X$ be a maximal independent set of $G$.
  In order to prove that $G$ is unmixed, 
  it is sufficient to show that $\# X = m$. 
  Set 
  \begin{displaymath}
    \begin{aligned}
      E_1 &:= \{ \{ x_{j1}, x_{j2} \} \in E(G) \; : \; 
                 \text{$\deg_G x_{j1} = 1$ or $\deg_G x_{j2} = 1$} \}, \\
      E_2 &:= \{ \{ x_{j1}, x_{j2} \} \in E(G) \; : \; 
                 \text{$\deg_G x_{j1} \geq 2$ and $\deg_G x_{j2} \geq 2$} \} 
    \end{aligned}
  \end{displaymath}
  and $V_k := \bigcup_{e \in E_k} e$ for $k=1,2$. 
  Since $V = V_1 \sqcup V_2 \sqcup W$, 
  \begin{displaymath}
    \# X = \# (X \cap V_1) + \# (X \cap (V_2 \cup W)). 
  \end{displaymath}

  \par
  Take $\{ x_{j1}, x_{j2} \} \in E_1$. 
  Assume that $\deg_G x_{j1} = 1$. 
  Then $\deg_G x_{j2} \geq 2$ because $G$ is connected. 
  If $x_{j2} \notin X$, then $x_{j1} \in X$ because the maximality of $X$. 
  This implies $\# (X \cap \{ x_{j1}, x_{j2} \}) = 1$ and 
  $\# (X \cap V_1) = m_1$. 

  \par
  By assumption, $V_2 \sqcup W$ can be decomposed as 
  $\bigcup_{i=1}^r \{ y_i, x_{j_i 1}, x_{j_i 2} \}$. 
  We claim that 
  $X \cap \{ y_i, x_{j_i 1}, x_{j_i 2} \} = 1$ for each $i$. 
%  Actually assume that $y_i \notin X$. 

  \par
  Since $\{ y_i, x_{j_i 1} \}, \{ y_i, x_{j_i 2} \}, 
  \{ x_{j_i 1}, x_{j_i 2} \} \in E(G)$, 
  it follows that 
  $\# (X \cap \{ y_i, x_{j_i 1}, x_{j_i 2} \}) \leq 1$. 
  Assume that $y_i \notin X$. 
  If $\deg y_i \geq 3$, then $\deg_G x_{j_i 1} = \deg_G x_{j_i 2} = 2$  
  and it follows from the maximality of $X$ 
  that exactly one of $x_{j_i 1}, x_{j_i 2}$ 
  belongs to $X$. 
  When $\deg y_i = 2$, if neither $x_{j_i 1}$ nor $x_{j_i 2}$ belong to 
  $X$, then $X \cup \{ y_i \}$ is also an independent set. 
  This contradicts to the maximality of $X$. 
  Thus (exactly) one of $x_{j_i 1}, x_{j_i 2}$ belongs to $X$. 
  Hence $\# (X \cap (V_2 \cup W)) = m_2 = \# W$. 

  \par
  Therefore 
  \begin{displaymath}
    \# X = \# (X \cap V_1) + \# (X \cap (V_2 \cup W)) 
         = m_1 + m_2 = m, 
  \end{displaymath}
  as desired. 
\end{proof}

\par
We show an example satisfying the assumption of Theorem \ref{DIM-VD}. 
\begin{example}
The graph $G$ in Figure \ref{fig:CMDIM} has a dominating induced matching, 
which is not a Cameron--Walker graph. 
It also satisfies the assumption of Theorem \ref{DIM-VD} 
with the displayed decomposition of the vertex set. 

\par
Then $m_2 = \# W = 3$. 
We can also easily see that this graph satisfies the assumption 
for the vertex in $W$ of Theorem \ref{DIM-VD-unmixed}. 
Hence $G$ is Cohen--Macaulay by Theorem \ref{DIM-VD-unmixed}. 
\end{example}

\begin{figure}[htbp]
  \begin{center}
    \includegraphics[width=120mm]{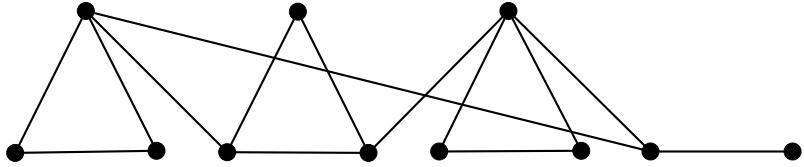}
  \end{center}
  \caption{A Cohen--Macaulay graph with a dominating induced matching}
  \label{fig:CMDIM}
\end{figure}

\par
We close the paper by giving some more examples of a graph with 
a dominating induced matching 
which does not satisfy the assumption on Theorem \ref{DIM-VD}. 

\par
We first show some Cohen--Macaulay graphs with a dominating induced matching. 
\begin{example}
  \begin{enumerate}
  \item The path graph $P_4$ with $4$ vertices 
    is a Cohen--Macaulay graph. Also it has a dominating induced matching. 
    Indeed, set $V(P_4) = \{ 1,2,3,4 \}$ and 
    $E(P_4) = \{ \{ 1,2 \}, \{ 2,3 \}, \{ 3,4 \} \}$. 
    Then we see that $P_4$ has a dominating induced matching with 
    the decomposition $V(P_4) = W \sqcup M$ where 
    $W = \{ 1,4 \}$ and $M = \{ 2,3 \}$. 
  \item The  graph $G_1$ on the vertex set $\{ 1,2,3,4,5,6 \}$ 
    whose edge set is 
    \begin{displaymath}
      E(G_1) 
        = \{ \{ 1,3 \}, \{ 1,4 \}, \{ 1,5 \}, \{ 2,4 \}, \{ 2,5 \}, \{ 2,6 \}, 
             \{ 3,4 \}, \{ 5,6 \} \}
    \end{displaymath}
    is a graph with a dominating induced matching. 
    This is Cohen--Macaulay, in particular, unmixed. 
    \newline
    \begin{picture}(100,60)
      \put(0,45){$G_1$:}
      \put(20,20){\circle*{5}}
      \put(40,20){\circle*{5}}
      \put(60,20){\circle*{5}}
      \put(80,20){\circle*{5}}
      \put(40,40){\circle*{5}}
      \put(60,40){\circle*{5}}
      %%%%%%%%%%%%%%%%%%%%%%%%
      \put(15,10){$3$}
      \put(40,10){$4$}
      \put(60,10){$5$}
      \put(80,10){$6$}
      \put(35,45){$1$}
      \put(60,45){$2$}
      %%%%%%%%%%%%%%%%%%%%%%%%
      \put(40,40){\line(-1,-1){20}}
      \put(40,40){\line(0,-1){20}}
      \put(40,40){\line(1,-1){20}}
      \put(60,40){\line(-1,-1){20}}
      \put(60,40){\line(0,-1){20}}
      \put(60,40){\line(1,-1){20}}
      \put(40,20){\line(-1,0){20}}
      \put(60,20){\line(1,0){20}}
    \end{picture}
  \end{enumerate}
\end{example}

We next show an unmixed graph with a dominating induced matching 
but not Cohen--Macaulay. 
\begin{example}
  The graph $G_2$ on the vertex set $\{ 1,2,3,4,5,6 \}$ whose edge set is 
  \begin{displaymath}
    \{ \{ 1,3 \}, \{ 1,4 \}, \{ 1,5 \}, \{ 1,6 \}, \{ 2,3 \}, \{ 2,4 \}, 
       \{ 2,5 \}, \{ 2,6 \}, \{ 3,4 \}, \{ 5,6 \} \}
  \end{displaymath}
  is a graph with a dominating induced 
  matching. 
  This is unmixed but not Cohen--Macaulay. 
  \newline
  \begin{picture}(100,60)
    \put(0,45){$G_2$:}
    \put(20,20){\circle*{5}}
    \put(40,20){\circle*{5}}
    \put(60,20){\circle*{5}}
    \put(80,20){\circle*{5}}
    \put(40,40){\circle*{5}}
    \put(60,40){\circle*{5}}
    %%%%%%%%%%%%%%%%%%%%%%%%
    \put(15,10){$3$}
    \put(40,10){$4$}
    \put(60,10){$5$}
    \put(80,10){$6$}
    \put(35,45){$1$}
    \put(60,45){$2$}
    %%%%%%%%%%%%%%%%%%%%%%%%
    \put(40,40){\line(-1,-1){20}}
    \put(40,40){\line(0,-1){20}}
    \put(40,40){\line(1,-1){20}}
    \put(60,40){\line(-1,-1){20}}
    \put(60,40){\line(0,-1){20}}
    \put(60,40){\line(1,-1){20}}
    \put(40,20){\line(-1,0){20}}
    \put(60,20){\line(1,0){20}}
    \put(40,40){\line(2,-1){40}}
    \put(60,40){\line(-2,-1){40}}
  \end{picture}
\end{example}

Finally, we show a sequentially Cohen--Macaulay graph with a 
dominating induced matching but not unmixed. 
\begin{example}
  The graph $G_3$ on the vertex set $\{ 1,2,3,4,5,6 \}$ whose edge set is 
  \begin{displaymath}
    \{ \{ 1,3 \}, \{ 3,4 \}, \{ 2,4 \}, \{ 2,5 \}, \{ 2,6 \}, \{ 5,6 \} \}
  \end{displaymath}
  is a graph with a dominating induced matching. 
  \newline
  \begin{picture}(100,60)
    \put(0,45){$G_3$:}
    \put(20,20){\circle*{5}}
    \put(40,20){\circle*{5}}
    \put(60,20){\circle*{5}}
    \put(80,20){\circle*{5}}
    \put(40,40){\circle*{5}}
    \put(60,40){\circle*{5}}
    %%%%%%%%%%%%%%%%%%%%%%%%
    \put(15,10){$3$}
    \put(40,10){$4$}
    \put(60,10){$5$}
    \put(80,10){$6$}
    \put(35,45){$1$}
    \put(60,45){$2$}
    %%%%%%%%%%%%%%%%%%%%%%%%
    \put(40,40){\line(-1,-1){20}}
    %  \put(40,40){\line(0,-1){20}}
    %  \put(40,40){\line(1,-1){20}}
    \put(60,40){\line(-1,-1){20}}
    \put(60,40){\line(0,-1){20}}
    \put(60,40){\line(1,-1){20}}
    \put(40,20){\line(-1,0){20}}
    \put(60,20){\line(1,0){20}}
  \end{picture}
  \newline
  This is not unmixed since 
  both $\{ 3,4,5,6 \}$ and $\{ 2,3,6 \}$ 
  are minimal vertex covers of $G_3$. 
  Also, since $G_3$ is chordal, it is sequentially Cohen--Macaulay 
  by Francisco and Van Tuyl \cite{FVT}. 
\end{example}

%%%%%%%%%%%%%%%%%%%%%%%%%%%%%%%%%%%%%%%%
\begin{acknowledgement}
  The third author is partially supported by JSPS Grant-in-Aid 
  for Young Scientists (B) 24740008. 

  \par
  We thank anonymous referees for reading the manuscript carefully. 
\end{acknowledgement}

%%%%%%%%%%%%%%%%%%%%%%%%%%%%%%%%%%%%%%%%
%%%%%%%%%%%%%%%%%%%%%%%%%%%%%%%%%%%%%%%%
%%% references
%\bibliographystyle{amsplain}

\end{document}